\documentclass[12pt]{amsart}

\usepackage{amsmath,amsthm,amssymb,mathdots}
\usepackage{setspace,shuffle}
\usepackage{tikz-cd}
\usepackage{color}
\allowdisplaybreaks
\numberwithin{equation}{section}
\newtheorem{thm}{Theorem}[section]
\newtheorem{prop}[thm]{Proposition}
\newtheorem{lem}[thm]{Lemma}
\newtheorem{cor}[thm]{Corollary}
\newtheorem{conj}[thm]{Conjecture}

\theoremstyle{definition}

\newtheorem{defn}[thm]{Definition}
\newtheorem{rem}[thm]{Remark}

\newcommand{\R}{\mathbb R}
\newcommand{\Z}{{\mathbb Z}}
\newcommand{\N}{{\mathbb N}}
\newcommand{\C}{{\mathbb C}}
\newcommand{\Q}{{\mathbb Q}}
\newcommand{\F}{{\mathbb F}}

\DeclareMathOperator{\wt}{wt}
\DeclareMathOperator{\len}{len}
\newcommand{\kk}{{\bf k}}
\newcommand{\el}{{\bf l}}

\newcommand{\qsh}{\shuffle_\hbar}

\newcommand{\MTWA}{\mathcal{Z}^{\mathcal{A},\omega}}
\newcommand{\MTWS}{\mathcal{Z}^{\mathcal{S},\omega}}
\newcommand{\MTWB}{\mathcal{Z}^{\bullet,\omega}}
\newcommand{\MTW}{\mathcal{Z}^{\omega}}

\title{Finite and symmetric Mordell-Tornheim \\
	multiple zeta values}

\author[H.~Bachmann]{Henrik Bachmann}
\address[H.~Bachmann]{Graduate School of Mathematics\\ Nagoya University\\
	Nagoya, Aichi 464-8602\\
	Japan}
\email{henrik.bachmann@math.nagoya-u.ac.jp}

\author[Y.~Takeyama]{Yoshihiro Takeyama}
\address[Y.~Takeyama]{Division of Mathematics\\
	Faculty of Pure and Applied Sciences\\
	University of Tsukuba\\
	Tsukuba, Ibaraki 305-8571\\
	Japan}
\email{takeyama@math.tsukuba.ac.jp}

\author[K.~Tasaka]{Koji Tasaka}
\address[K.~Tasaka]{Department of Information Science and Technology\\
	Aichi Prefectural University\\
	Nagakute-city, Aichi 480-1198\\
	Japan
}
\email{tasaka@ist.aichi-pu.ac.jp}



\begin{document}

	\maketitle

	\begin{abstract}
		We introduce finite and symmetric Mordell-Tornheim type of multiple zeta values and give a new approach to the Kaneko-Zagier conjecture stating that the finite and symmetric multiple zeta values satisfy the same relations.
	\end{abstract}


	\section{Introduction and Main results}


	In \cite{BTT18} the authors described a new approach to study the relationship between finite and symmetric multiple zeta values. This was done by viewing these values as an algebraic and analytic limit of certain $q$-series evaluated at roots of unity. The purpose of this note is to use this approach to introduce finite and symmetric Mordell-Tornheim multiple zeta values and to give an analogue of the Kaneko-Zagier conjecture, which gives a surprising relationship between these values.

	\subsection{Notations}
	In this work, we denote by $\mathbb{F}_p$ the finite field with $p$ elements and by $\N$ the set of positive integers.
	We call a tuple $\kk=(k_1,\ldots,k_r)$ of positive integers an \textit{index},
	$\wt(\kk)=k_1+\cdots+k_r$ the \textit{weight} and
	$\len(\kk)=r$ the \textit{length}.
	An index $\kk=(k_1,\ldots,k_r)$ satisfying $k_1\ge2$ is said to be \textit{admissible}.
	We regard the empty index $\emptyset$ as admissible, and let $\wt(\emptyset)=\len(\emptyset)=0$.
	We set $F(\emptyset)$ to be a unit element for any function $F$ on indices.

	\subsection{A review: the Kaneko-Zagier conjecture}
	Kaneko and Zagier \cite{KanekoZagier} introduced the finite multiple zeta value $\zeta_{\mathcal{A}}(\kk)$
	for an index $\kk=(k_1,\ldots,k_r)$ as an element in the ring
	$\mathcal{A}=\big(\prod_{p} \mathbb{F}_p\big) \big/ \big(\bigoplus_{p} \mathbb{F}_p\big) $
	with $p$ running over all prime (see Definition \ref{def:fmzv}).
	Since the ring $\mathcal{A}$ forms a $\Q$-algebra,
	one can consider the $\Q$-vector subspace $\mathcal{Z}_k^{\mathcal{A}}$ of $\mathcal{A}$
	spanned by all finite multiple zeta values of weight $k$. The identity
	$\dim_\Q \mathcal{Z}^{\mathcal{A}}_k\stackrel{?}{=} \dim_\Q \mathcal{Z}_k-\dim_\Q \mathcal{Z}_{k-2}$
	is numerically observed by Zagier,
	where $\mathcal{Z}_k$ denotes the $\Q$-vector space spanned by multiple zeta values of weight $k$.
	Here the multiple zeta value is defined for an admissible index $\kk=(k_1,\ldots,k_r)$ by
	\[\zeta(\kk) := \sum_{m_1>\cdots >m_r>0} \frac{1}{m_1^{k_1}\cdots m_r^{k_r}}.\]
	Inspired by Kontsevich's idea, Kaneko and Zagier are led into a real counterpart of the finite multiple zeta value, called the symmetric multiple zeta value and denoted by $\zeta_{\mathcal{S}}(\kk)$ (see Definition \ref{def:smzv}), in the quotient $\Q$-algebra $\mathcal{Z}/\zeta(2)\mathcal{Z}$, where $\mathcal{Z}=\sum_{k\ge0} \mathcal{Z}_k$.
	They observed that finite and symmetric multiple zeta values satisfy the same relations over $\Q$, and then proposed a one-to-one correspondence between them, which we call the Kaneko-Zagier conjecture (see Conjecture \ref{conj:Kaneko-Zagier}).

	\subsection{Finite Mordell-Tornheim multiple zeta value}
	In this paper, we wish to first mimic the above story replacing the multiple harmonic sum $H_n(\kk)$ defined in \eqref{eq:harmonic_sum} with the rational number $\omega_n(\kk)\in \Q$ defined for an index $\kk=(k_1,\ldots,k_r)$ and $n\in \N$ by
	\[\omega_n(\kk) := \sum_{\substack{m_1+\cdots+m_r=n\\m_1,\ldots,m_r>0}} \prod_{a=1}^r \frac{1}{m_a^{k_a}} ,\]
	and to then propose an alternative approach to the Kaneko-Zagier conjecture.
	Here the notion of $\omega_n(\kk)$ is originated from a certain truncation of generalized Mordell-Tornheim-Witten sums
\begin{align*}
\omega(\kk| \el) = \sum_{\substack{m_1+\cdots+m_r=n_1+\cdots+n_s\\ m_1,\ldots,m_r>0\\n_1,\ldots,n_s>0}}
\prod_{a=1}^r \frac{1}{m_a^{k_a}}\prod_{b=1}^s \frac{1}{n_b^{l_b}}
\end{align*}
studied by \cite{BaileyBorweinBorwein,BaileyBorwein1,BaileyBorweinCrandall,BaileyBorwein2}.
We start by defining a finite analogue of these values.

	\begin{defn}
		For an index $\kk=(k_1,\ldots,k_r)$ with $r\ge2$, we define the \emph{finite multiple omega value} $\omega_{\mathcal{A}}(\kk)$ by
		\[ \omega_{\mathcal{A}}(\kk) := \left(\omega_p(\kk)  \mod p \right)_p \in \mathcal{A}.\]
	\end{defn}

	Denote by $\MTWA_k$ the $\Q$-vector space spanned by all finite multiple omega values of weight $k$.
	Using PARI-GP \cite{PARI}, we numerically computed the dimension of $\MTWA_k$.
	For comparison, it is listed together with the conjectural dimension of $\mathcal{Z}_k$ as follows.
	\begin{center}
		\begin{tabular}{c|ccccccccccccccccccccccccc}
			$ k$&1&2&3&4&5&6&7&8 &9&10&11&12\\ \hline
			$\dim \MTWA_k$ & 0 & 0 & 1 & 0 & 1 & 1 & 1 & 2 & 2 & 3 & 4 & 5\\ \hline
			$\dim \mathcal{Z}_{k}$ & 0 & 1 & 1 & 1 & 2 & 2 & 3 & 4 & 5 & 7 & 9 & 12\\
		\end{tabular}
	\end{center}
	Similarly to the finite multiple zeta value, the equality $\dim_\Q \MTWA_k\stackrel{?}{=} \dim_\Q \mathcal{Z}_k-\dim_\Q \mathcal{Z}_{k-2}$ up to $k= 12$ is observed.
	Accordingly, we may expect a real counterpart $\omega_{\mathcal{S}}(\kk)$ of $\omega_{\mathcal{A}}(\kk)$ in $\mathcal{Z}/\zeta(2)\mathcal{Z}$.

	\subsection{Symmetric Mordell-Tornheim multiple zeta value}

	In order to find a proper definition of $\omega_{\mathcal{S}}(\kk)$, we use the same framework with as used in our previous work \cite{BTT18}.
	A crucial feature of this work was that finite and symmetric multiple zeta values are obtained from an algebraic and analytic operation for multiple harmonic $q$-series $z(\kk;q)$ defined in  \eqref{eq:def_znq}, which is a $q$-analogue of $H_n(\kk)$, at primitive roots of unity (see Theorem \ref{thm:BTT}).
	The Kaneko-Zagier conjecture then turns out to be the statement of a relationship between `values' of $z(\kk;q)$ at $q=\zeta_n$ with $\zeta_n$ being a primitive $n$-th root of unity.

	Let us begin with the $q$-analogue of $\omega_n(\kk)$.
	For an index $\kk=(k_1,\ldots,k_r)$, let
	\begin{equation}\label{eq:def_qMTW}
	\omega_n(\kk;q):= \sum_{\substack{m_1+\cdots+m_r=n\\m_1,\ldots,m_r>0}} \prod_{a=1}^r \frac{q^{(k_a-1)m_a}}{[m_a]^{k_a}} ,
	\end{equation}
	where $[m]=1+q+\cdots+q^{m-1}=\frac{1-q^m}{1-q}$ denotes the usual $q$-integer.
	By definition, it follows that $\omega_n(k_{\sigma(1)},\ldots,k_{\sigma (r)};q)= \omega_n(k_1,\ldots,k_r;q)$ for any permutation $\sigma\in \mathfrak{S}_r$.
	If $\len(\kk)\ge2$, one can consider the value $\omega_n(\kk;\zeta_n)$ in the cyclotomic field $\Q(\zeta_n)$ at $q=\zeta_n$ a primitive $n$-th root of unity (note that $\omega_n(\kk;\zeta_n)$ is not well-defined if $\zeta_n$ is not primitive).
	Its connection with the finite multiple omega value is as follows.

	\begin{thm}\label{thm:FMZV}
		Let $\zeta_p$ be a primitive $p$-th root of unity.
		Under the identification $\Z[\zeta_p]/(1-\zeta_p)\Z[\zeta_p]= \F_p$ with $p$ being prime, for any index $\kk$ with $\len(\kk)\ge2$ we have
		\[ \left(\omega_p(\kk;\zeta_p) \mod (1-\zeta_p)\Z[\zeta_p]\right)_p = \omega_{\mathcal{A}}(\kk) \in \mathcal{A}.\]
	\end{thm}

	Theorem \ref{thm:FMZV} says that the finite multiple omega value is obtained from a certain algebraic substitution $q\rightarrow 1$ for $\omega_n (\kk;q)$.
	On the other hand, by an analytic limit $q\rightarrow 1$ along the unit circle, we are led into the definition of a real counterpart of finite multiple omega value.

	\begin{thm}\label{thm:SMZV}
		For any index $\kk=(k_1,\ldots,k_r)$ with $r\ge2$, the limit
		\[ \Omega(\kk) := \lim_{n\rightarrow \infty} \omega_n(\kk;e^{\frac{2\pi i}{n}})\]
		exists, and it is given by
		\begin{equation}\label{eq:omega_formula}
		\Omega(\kk) = \sum_{a=1}^r (-1)^{k_a} \zeta^{MT}(\underbrace{k_1,\ldots,k_{a-1}}_{a-1},\underbrace{k_{a+1},\ldots,k_r}_{r-a};k_a),
		\end{equation}
		where $\zeta^{MT}(k_1,\ldots,k_r;l)$ is the Mordell-Tornheim multiple zeta value defined in \eqref{eq:def_MT}.
	\end{thm}

	Since every Mordell-Tornheim multiple zeta value can be written as $\Q$-linear combinations of multiple zeta values (see \cite[Theorem 1.1]{BradleyZhou}), the following definition makes sense.

	\begin{defn}
		For an index $\kk=(k_1,\ldots,k_r)$ with $r\ge2$, we define the \emph{symmetric multiple omega value} $\omega_{\mathcal{S}}(\kk)$ by
		\[ \omega_{\mathcal{S}}(\kk) := \Omega(\kk) \mod \zeta(2)\mathcal{Z}.\]
	\end{defn}

	Using Mathematica \cite{Mathematica}, we also numerically checked that the identity $\dim_\Q \MTWS_k \stackrel{?}{=} \dim_\Q \mathcal{Z}_k - \dim_\Q \mathcal{Z}_{k-2}$ holds up to weight 12, where $\MTWS_k$ denotes the $\Q$-vector subspace of $\mathcal{Z}/\zeta(2)\mathcal{Z}$ spanned by all symmetric multiple omega values of weight $k$.
	Consequently, for all $k\in \N$, the space $\MTWS_k$ is expected to be isomorphic to $\MTWA_k$ as a $\Q$-vector space.

	\subsection{Kaneko-Zagier conjecture revisited}

	We compared all $\Q$-linear relations among finite and symmetric multiple omega values up to weight 12, and it is observed that they satisfy the same relations over $\Q$.
	This further observation will support the following analogous statement to the Kaneko-Zagier conjecture.

	\begin{conj}\label{conj:main}
		For any $k\in \N$, the $\Q$-linear map $\varphi_k: \MTWA_k\rightarrow \MTWS_k$ given by
		\[ \varphi_k (\omega_{\mathcal{A}}(\kk) ) = \omega_{\mathcal{S}}(\kk)\]
		is a well-defined isomorphism.
	\end{conj}

	Another evidence of Conjecture \ref{conj:main} will be provided as follows.
	Let $\mathfrak{h}=\Q\langle x_0,x_1\rangle$ and $\mathfrak{h}^1=\Q+\mathfrak{h}x_1$.
	We write $y_k=x_0^{k-1}x_1$ for $k\in \N$ and $y_{\kk}=y_{k_1}\cdots y_{k_r}$ for an index $\kk=(k_1,\ldots,k_r)$.
	Then, the set $\{y_\kk \mid \kk :\mbox{index}\}$ is a linear basis of $\mathfrak{h}^1$.
	For $\bullet \in \{\mathcal{A},\mathcal{S}\}$, we write $\zeta_{\mathcal{\bullet}}(y_{\kk})=\zeta_{\mathcal{\bullet}}(\kk)$, and extend these notations to $\mathfrak{h}^1$ by $\Q$-linearly. 
	Denote by $\shuffle: \mathfrak{h}\times \mathfrak{h}\rightarrow \mathfrak{h}$ the standard shuffle product (see \eqref{eq:def_shuffle}).

	\begin{thm}\label{thm:omega_fsmzv}
		For any index $\kk=(k_1,\ldots,k_r)$ with $r\ge2$, we have
		\begin{align}\label{eq:omega_fmzv}
		\omega_{\mathcal{A}}({\kk}) = (-1)^{k_r} \zeta_{\mathcal{A}}(x_0^{k_r} (y_{k_1}\shuffle \cdots \shuffle y_{k_{r-1}})),\\
		\label{eq:omega_smzv}
		\omega_{\mathcal{S}}({\kk}) = (-1)^{k_r} \zeta_{\mathcal{S}}(x_0^{k_r} (y_{k_1}\shuffle \cdots \shuffle y_{k_{r-1}})).
		\end{align}
	\end{thm}

	By Theorem \ref{thm:omega_fsmzv}, we see that the Kaneko-Zagier conjecture, which insists the equality $\ker \zeta_{\mathcal{A}} \stackrel{?}{=} \ker \zeta_{\mathcal{S}}$, implies Conjecture \ref{conj:main}.
	Numerically, we also observed that all finite and symmetric multiple zeta values can be written as $\Q$-linear combinations of finite and symmetric multiple omega values, respectively.
	If this is true, then Conjecture \ref{conj:main} would imply the map $\varphi_{KZ}$ defined in Conjecture \ref{conj:Kaneko-Zagier} (the Kaneko-Zagier conjecture) induces a well-defined isomorphism of $\Q$-vector spaces.
	Thus, Conjecture \ref{conj:main} could be a new approach to the Kaneko-Zagier conjecture.
	An advantage of this point of view is that we do not need to regularize our values, but a weak point is that we can not describe the algebraic structure of the $\omega_{\bullet}$.

	We remark that Kamano \cite{Kamano} introduces the finite Mordell-Tornheim multiple zeta value $\zeta^{MT}_{\mathcal{A}}(\kk)$ for each index $\kk$.
	By definition, it follows that $\omega_{\mathcal{A}}(\kk)=(-1)^{k_r} \zeta^{MT}_{\mathcal{A}}(\kk)$ for any index $\kk$.
	Then, Theorem 1.2 of \cite{Kamano} is equivalent to \eqref{eq:omega_fmzv},
	but our proof is completely different (see Remark \ref{rem:Kamano}).

	\subsection{Contents}

	The organization of this paper is as follows.
	In Section 2, we recall some basics on the Kaneko-Zagier conjecture and Mordell-Tornheim multiple zeta values, thereby also fixing some of our notation.
	Section 3 is devoted to proving Theorems \ref{thm:FMZV} and \ref{thm:SMZV} and Section 4 gives a proof of Theorem \ref{thm:omega_fsmzv}.
	Our proof of \eqref{eq:omega_smzv} provides another expression of the limit value $\Omega(\kk)$ in terms of $\zeta_{\mathcal{S}}^\shuffle (\kk)$'s, which is mentioned in the end of Section 4.
	Sections 5 considers relations among the values $\omega_n(\kk;\zeta_n)$, which is applied to the study of relations of finite and symmetric multiple omega values together with special values of them.\\

	\noindent\textbf{Acknowledgments.}
	The authors would like to thank Masataka Ono, Shin-ichiro Seki and Shuji Yamamoto for kindly sharing their results on symmetric Mordell-Tornheim multiple zeta values \cite{OSY} and for pointing out some mistakes. 
	This work was partially supported by JSPS KAKENHI Grant Numbers 18K13393, 18K03233, 19K14499.

	\section{Preliminaries}

	\subsection{Statement of the Kaneko-Zagier conjecture}\label{subsection:KZconj}
	We briefly review the Kaneko-Zagier conjecture.

	We define the finite multiple zeta value.
	It is an element in the ring $\mathcal{A}$, which is introduced by Kontsevich \cite[\S2.2]{Kontsevich}, defined by
	\begin{align*}
	\mathcal{A} = \left( \prod_{p:{\rm prime}} \F_p \right) \big/ \left( \bigoplus_{p:{\rm prime}} \F_p \right)\,.
	\end{align*}
	Its element is denoted by $(a_p)_p$, where $p$ runs over all primes and $a_p\in \F_p$.
	Two elements $(a_p)_p$ and $(b_p)_p$ are identified if and only if $a_p=b_p$ for all but finitely many primes $p$.
	Rational numbers $\Q$ can be embedded into $\mathcal{A}$ as follows.
	For $a\in \Q$, set $a_p=0$ if $p$ divides the denominator of $a$ and $a_p=a\in \F_p$ otherwise.
	Then $(a_p)_p \in \mathcal{A}$.
	In this way, the ring $\mathcal{A}$ forms a commutative algebra over $\Q$.

	Let us define the multiple harmonic sum $ H_n(\kk)\in\Q$ for an index $\kk=(k_1,\ldots,k_r)$ and $n\in\N$ by
	\begin{equation}\label{eq:harmonic_sum}
	H_n(\kk)=\sum_{n>m_1>\cdots >m_r>0} \frac{1}{m_1^{k_1}\cdots m_r^{k_r}}.
	\end{equation}

	\begin{defn}\label{def:fmzv}
		For an index $\kk=(k_1,\ldots,k_r)$, we define the finite multiple zeta value $\zeta_{\mathcal{A}} (\kk)$ by
		\begin{align*}
		\zeta_{\mathcal{A}} (\kk)=
		\left(H_p(\kk) \mod p \right)_p \in \mathcal{A} .
		\end{align*}
	\end{defn}
	Denote by $\mathcal{Z}^{\mathcal{A}}_k$ the $\Q$-vector subspace of $\mathcal{A}$ spanned by all finite multiple zeta values of weight $k$ and set $\mathcal{Z}^{\mathcal{A}}=\sum_{k\ge0}\mathcal{Z}^{\mathcal{A}}_k$, which forms a $\Q$-algebra (see \cite{Kaneko,KanekoZagier}).

	\

	Let us turn to the symmetric multiple zeta values.
	We first recall the algebraic setup of multiple zeta values by Hoffman \cite{Hoffman1}.
	Let
	\[\mathfrak{h}=\mathbb{Q}\langle x_0,x_1 \rangle\]
	be the non-commutative polynomial ring
	with indeterminates $x_0$ and $x_1$, and set $y_{k}=x_0^{k-1}x_1$ $(k \ge 1)$.
	For a non-empty index $\kk=(k_{1}, \ldots , k_{r})$ we set $y_{\kk}=y_{k_{1}}\cdots y_{k_{r}}$.
	We define $y_{\emptyset}=1$ for the empty index.
	We also let
	\[\mathfrak{h}^{1}=\Q+\mathfrak{h}x_1,\quad \mathfrak{h}^{0}=\Q+x_0\mathfrak{h}x_1.\]
	The $\Q$-vector subspace $\mathfrak{h}^{1}$ is the $\Q$-subalgebra freely generated by $\{y_{k}\}_{k \ge 1}$, and the set of monomials $y_{\kk}$ with admissible index $\kk$ is a linear basis of $\mathfrak{h}^0$.

	Using the iterated integral expression of the multiple zeta value due to Kontsevich, one can prove that the map $\zeta:\mathfrak{h}^0\rightarrow \R$, defined by $\zeta(y_{\kk})=\zeta(\kk)$ for an admissible index $\kk$, is an algebra homomorphism with respect to the shuffle product $\shuffle : \mathfrak{h} \times \mathfrak{h} \to \mathfrak{h}$ given inductively by
	\begin{equation}\label{eq:def_shuffle}
	uw\shuffle vw'=u(w \shuffle vw')+v(uw \shuffle w')
	\end{equation}
	for $w,w' \in \mathfrak{h}$ and $u, v \in \{x_0, x_1\}$, with the initial condition $w \shuffle 1=w=1\shuffle w$.
	Namely, $\zeta(w_{1})\zeta(w_{2})=\zeta(w_{1} \shuffle w_{2})  $ holds for any $w_{1}, w_{2} \in \mathfrak{h}^{0}$.
	Equipped with the shuffle product, the vector space $\mathfrak{h}$ forms a commutative $\Q$-algebra and $\mathfrak{h}^1,\mathfrak{h}^0$ are $\Q$-subalgebras.
	We write $\mathfrak{h}_\shuffle,\mathfrak{h}^1_\shuffle,\mathfrak{h}^0_\shuffle$ for commutative $\Q$-algebras with the shuffle product.
	It follows that $\zeta(\mathfrak{h}^0)=\mathcal{Z}$, which is hence a $\Q$-subalgebra of $\R$.

	Let us define the shuffle regularized multiple zeta values, following \cite{IKZ}.
	Recall that the algebra $\mathfrak{h}^1_\shuffle$ is freely generated by $x_{1}$
	over $\mathfrak{h}^0_\shuffle$ (see \cite{Reutenauer}):
	\[ \mathfrak{h}^1_\shuffle \cong \mathfrak{h}^0_\shuffle[x_1].\]
	Namely, for any word $w\in \mathfrak{h}^1$, there exist $w_i\in \mathfrak{h}^0$ such that
	\[ w= w_0+w_1\shuffle x_1+w_2\shuffle x_1^{\shuffle 2}+\cdots+ w_n \shuffle x_1^{\shuffle n}.\]
	Hence, there is a unique algebra homomorphism $\zeta^\shuffle: \mathfrak{h}^1_\shuffle\rightarrow \R[T]$ such that the map $\zeta^\shuffle$ extend $\zeta:\mathfrak{h}^0_\shuffle\rightarrow \R$ and send $\zeta^\shuffle(x_1)=T$.
	Applying $\zeta^\shuffle$ to the above word, we get
	\[ \zeta^\shuffle(w)=\sum_{a=0}^n \zeta(w_a) T^a,\]
	so by definition it follows that $ \zeta^\shuffle(\mathfrak{h}^1)=\mathcal{Z}[T]$.
	For an index $\kk =(k_1,\ldots,k_r) $ we write $\zeta^\shuffle(y_\kk)\big|_{T=0}=\zeta^\shuffle(\kk)\in \mathcal{Z}$, which we call the shuffle regularized multiple zeta value.
	To define the symmetric multiple zeta value, for an index $\kk=(k_1,\ldots,k_r)$, let
	\begin{equation}\label{eq:def_zeta_s}
	\zeta^\shuffle_{\mathcal{S}}(\kk)= \sum_{a=0}^r (-1)^{k_1+\cdots+k_a} \zeta^\shuffle(k_a,k_{a-1},\ldots,k_1)\zeta^\shuffle(k_{a+1},k_{a+2},\ldots,k_r) .
	\end{equation}

	\begin{defn}\label{def:smzv}
		For an index $\kk=(k_1,\ldots,k_r)$, we define the symmetric multiple zeta value $ \zeta_{\mathcal{S}}(\kk)$ by
		\[ \zeta_{\mathcal{S}}(\kk):=\zeta^\shuffle_{\mathcal{S}}(\kk)\mod \zeta(2)\mathcal{Z} ,\]
		which is an element in the quotient $\Q$-algebra $\mathcal{Z}/\zeta(2)\mathcal{Z}$.
	\end{defn}

	We remark that Yasuda \cite[Theorem 6.1]{Yasuda} proved that the symmetric multiple zeta value $\zeta_{\mathcal{S}}(\kk)$ spans the whole space $\mathcal{Z}/\zeta(2)\mathcal{Z}$.
	Remark that there is another variant of $\zeta_{\mathcal{S}}^\shuffle(\kk)$ replacing $\zeta^\shuffle$ with $\zeta^\ast$ the harmonic regularized multiple zeta value in the definition, which is however known to be congruent to $\zeta_{\mathcal{S}}(\kk)$ modulo $\zeta(2)$ (see \cite[Proposition 9.1]{Kaneko} and \cite{KanekoZagier}).

	\

	The Kaneko-Zagier conjecture is stated as follows.

	\begin{conj}\label{conj:Kaneko-Zagier}
		The $\Q$-linear map $\varphi_{KZ}: \mathcal{Z}^{\mathcal{A}}\rightarrow \mathcal{Z}/\zeta(2)\mathcal{Z}$ given by
		\[ \varphi_{KZ}(\zeta_{\mathcal{A}}(\kk)) = \zeta_{\mathcal{S}}(\kk) \]
		is a well-defined isomorphism of $\Q$-algebras.
	\end{conj}

	\subsection{Mordell-Tornheim multiple zeta value}\label{sec:q}
	In this subsection, we recall the Mordell-Tornheim type of multiple zeta values.

	The Mordell-Tornheim multiple zeta value is defined for an index $\kk=(k_1,\ldots,k_r)$ and $l\in\N$ by
	\begin{equation}\label{eq:def_MT}
	\zeta^{MT}(\kk;l) = \sum_{m_1,\ldots,m_r>0} \frac{1}{m_1^{k_1}\cdots m_r^{k_r} (m_1+\cdots+m_r)^l}.
	\end{equation}
	This type of sum for the case $r=2$ was first studied by Tornheim \cite{Tornheim} in 1950, and independently by Mordell \cite{Mordell} in 1958 with $k_1=k_2=l$, and then, rediscovered by Witten \cite{Witten} in 1991 in his volume formula for certain moduli spaces related to theoretical physics (see also Zagier's number theoretical treatment \cite{Zagier}).
	The case $r\ge3$ was first introduced by Matsumoto \cite{Matsumoto} in 2002 as a function of several complex variables, and then investigated by many authors from many point of views (see e.g. \cite{MatsumotoNakamuraOchiaiTsumura,Nakamura,Onodera,Tsumura02,Tsumura05,Tsumura07}).

	As a study of special values, it was shown by Bradley-Zhou \cite[Theorem 1.1]{BradleyZhou} in 2010 that every Mordell-Tornheim multiple zeta value can be written as a $\Q$-linear combination of multiple zeta values.
	Moreover, for an index $\kk=(k_1,\ldots,k_r)$ Yamamoto \cite{Yamamoto} proved the following formula:
	\begin{equation}\label{eq:MT_mz}
	\zeta^{MT}(k_1,\ldots,k_{r-1};k_r) =\zeta(x_0^{k_r} (y_{k_1}\shuffle \cdots \shuffle y_{k_{r-1}})).
	\end{equation}

	It is worth mentioning that there seems no study on the opposite implication;
	``Every multiple zeta value can be written as a $\Q$-linear combination of Mordell-Tornheim multiple zeta values''.
	We numerically checked this implication up to weight 12 using Mathematica:
	\begin{conj}\label{conj:MT}
		The space $\mathcal{Z}$ is generated by Mordell-Tornheim multiple zeta values over $\Q$.
	\end{conj}

	Remark that to check Conjecture \ref{conj:MT}, it is necessary to use relations among multiple zeta values.
	In other words, one can check that there exist a monomial $y_\kk$ such that it cannot be written as $\Q$-linear combinations of $x_0^{k_r} (y_{k_1}\shuffle \cdots \shuffle y_{k_{r-1}})$'s.

	\


	\section{Proofs of Theorems \ref{thm:FMZV} and \ref{thm:SMZV}}
	\subsection{Proof of Theorem \ref{thm:FMZV}}
	We prove Theorem \ref{thm:FMZV}.

	\begin{proof}[Proof of Theorem \ref{thm:FMZV}]
		We first notice that the value $[m]\big|_{q={\zeta_p}}=(1-\zeta_p^m)/(1-\zeta_p)$ for any prime $p$ is a cyclotomic unit.
		Hence, the value $\omega_p (\kk;\zeta_p)$ lies in $\Z[\zeta_p]$, and taking modulo the ideal $(1-\zeta_p)\Z[\zeta_p]$ generated by $1-\zeta_p$ in $\Z[\zeta_p]$ makes sense.
		Since $(1-\zeta_p)\Z[\zeta_p]$ is a prime ideal,
		its residue field $\Z[\zeta_p]/(1-\zeta_p)\Z[\zeta_p]$ is $\mathbb{F}_p$.
		It follows that
		$q^{(k-1)m}/[m]^{k}|_{q=\zeta_p}\equiv m^{-k} \mod{(1-\zeta_p)\Z[\zeta_p]}$ for all $k\in \N$,
		and hence
		\[ \omega_p(\kk;\zeta_p) \equiv \omega_p(\kk) \mod{(1-\zeta_p)\Z[\zeta_p]} .\]
		Then, the result follows from the definition of the finite multiple omega value.
	\end{proof}

	\subsection{Proof of Theorem \ref{thm:SMZV}}
	We now prove Theorem \ref{thm:SMZV}.

	\begin{proof}[Proof of Theorem \ref{thm:SMZV}]
		We see that
		\begin{align*}
		\frac{1}{[m]}\bigg|_{q=e^{2\pi i/n}}=e^{-\frac{\pi i}{n}(m-1)}\frac{\sin{\frac{\pi}{n}}}{\sin{\frac{m \pi}{n}}} \qquad
		(n>m > 0).
		\end{align*}
		Therefore it holds that
		\begin{align*}
		\omega_{n}(\kk; e^{2\pi i/n})=\left(e^{\frac{\pi i}{n}}\frac{n}{\pi}\sin{\frac{\pi}{n}}\right)^{\wt(\kk)}\,
		G_{n}(\kk),
		\end{align*}
		where the function $G_{n}(\kk)$ is defined by
		\begin{align*}
		G_{n}(\kk)=\frac{1}{n^{\wt(\kk)}}
		\sum_{\substack{m_{1}, \ldots , m_{r}>0 \\ m_{1}+\cdots +m_{r}=n}}
		\prod_{a=1}^{r}g_{k_{a}}\left(\frac{m_{a}}{n}\right), \qquad
		g_{k}(x)=\left(\frac{\pi}{\sin{\pi x}}\right)^{k}e^{(k-2) \pi i x} \,\, (k \in \N).
		\end{align*}
		It suffices to show that $G_{n}(\kk)$ converges to the right side of \eqref{eq:omega_formula}
		in the limit $n \to \infty$.
		We decompose the range of the sum
		$I=\{ (m_{1}, \ldots , m_{r}) \in \N^{r} \, | \, \sum_{a=1}^{r}m_{a}=n \}$
		into the subsets
		\begin{align*}
		&
		I_{0}=\{
		(m_{1}, \ldots , m_{r}) \in I \, | \,  0<m_{a}\le n/2 \,\, (1\le \forall{a} \le r)
		\}, \\
		&
		I_{j}=\{ (m_{1}, \ldots , m_{r}) \in I \, | \,  m_{j}>n/2, \, 0<m_{a}\le n/2 \,\, (\forall{a}\not=j) \},
		\quad (1\le j \le r).
		\end{align*}
		Note that $I=\sqcup_{j=0}^{r}I_{j}$.
		Now we set
		\begin{align*}
		&
		A_{n}(l_{1}, \ldots , l_{r})=\frac{1}{n^{\sum_{a=1}^{r}l_{a}}}
		\sum_{\substack{n/2\ge m_{1}, \ldots , m_{r}>0 \\ m_{1}+\cdots +m_{r}=n}}
		\prod_{a=1}^{r}g_{l_{a}}\left(\frac{m_{a}}{n}\right), \\
		&
		B_{n}(l_{1}, \ldots , l_{r-1}; l_{r})=\frac{1}{n^{\sum_{a=1}^{r}l_{a}}}
		\sum_{\substack{m_{1}, \ldots , m_{r-1}>0 \\ m_{1}+\cdots +m_{r-1}<n/2}}
		\prod_{a=1}^{r-1}g_{l_{a}}\left(\frac{m_{a}}{n}\right)\,
		h_{l_{r}}\left(\frac{m_{1}+\cdots +m_{r-1}}{n}\right),
		\end{align*}
		where
		\begin{align*}
		h_{l}(x)=(-1)^{l}g_{l}(1-x)=\left(\frac{\pi}{\sin{\pi x}}\right)^{l}e^{-(l-2) \pi i x}.
		\end{align*}
		Then it holds that
		\begin{align*}
		&
		\frac{1}{n^{\wt(\kk)}}
		\sum_{(m_{1}, \ldots , m_{r}) \in I_{0}}\prod_{a=1}^{r}g_{k_{a}}\left(\frac{m_{a}}{n}\right)=
		A_{n}(k_{1}, \ldots , k_{r}), \\
		&
		\frac{1}{n^{\wt(\kk)}}
		\sum_{(m_{1}, \ldots , m_{r}) \in I_{j}}\prod_{a=1}^{r}g_{k_{a}}\left(\frac{m_{a}}{n}\right)=
		(-1)^{k_{j}}B_{n}(k_{1}, \ldots , k_{j-1},k_{j+1}, \ldots , k_{r}; k_{j}).
		\end{align*}
		Therefore it is enough to show that, for any index $\kk$ and $l \in \N$,
		\begin{align*}
		A_{n}(\kk) \to 0, \qquad \, B_{n}(\kk; l)\to \zeta^{MT}(\kk; l)
		\end{align*}
		as $n \to \infty$.

		First, we calculate the limit of $A_{n}(\kk)$.
		Set
		\begin{align*}
		\tilde{A}_{n}(\kk)=\sum_{\substack{n/2\ge m_{1}, \ldots , m_{r}>0 \\ m_{1}+\cdots +m_{r}=n}}
		\frac{1}{m_{1}^{k_{1}} \cdots m_{r}^{k_{r}}}.
		\end{align*}
		It holds that
		\begin{align}
		|g_{k}(x)|\le \left(\frac{\pi}{2x}\right)^{k} \qquad (0<x\le \frac{1}{2}).
		\label{eq:estimate1}
		\end{align}
		Hence $|A_{n}(\kk)| \le C_{\kk} \tilde{A}_{n}(\kk)$
		for a constant $C_{\kk}$ which does not depend on $n$.
		The limit of $\tilde{A}_{n}(\kk)$ as $n \to \infty$ is equal to zero because
		\begin{align*}
		0\le \tilde{A}_{n}(\kk) \le
		\sum_{\substack{n/2\ge m_{1}, \ldots , m_{r}>0 \\ m_{1}+\cdots +m_{r}=n}}
		\frac{1}{m_{1} \cdots m_{r}} \le
		\sum_{\substack{m_{1}, \ldots , m_{r}>0 \\ m_{1}+\cdots +m_{r}=n}}
		\left(\frac{2}{n}\right)^{r}=
		\left(\frac{2}{n}\right)^{r}\binom{n-1}{r-1} \to 0.
		\end{align*}
		Therefore $A_{n}(\kk) \to 0$ as $n \to 0$.

		Next we calculate the limit of $B_{n}(\kk; l)$
		for an index $\kk=(k_{1}, \ldots , k_{r-1})$ and a positive integer $l$.
		Set
		\begin{align*}
		\tilde{B}_{n}(\kk; l)=\sum_{\substack{m_{1}, \ldots , m_{r-1}>0 \\ m_{1}+\cdots +m_{r-1}<n/2}}
		\frac{1}{m_{1}^{k_{1}} \cdots m_{r-1}^{k_{r-1}} (m_{1}+\cdots +m_{r-1})^{l}}.
		\end{align*}
		Since the function $f(z)=z^{k-1}(g_{k}(z)-z^{-k})$ is regular in a neighborhood
		of the interval $[0, 1/2]$, there exists a positive constant $C$ such that
		\begin{align}
		\left| g_{k}(x)-\frac{1}{x^{k}} \right| \le \frac{C}{x^{k-1}} \qquad (0<x \le 1/2).
		\label{eq:estimate2}
		\end{align}
		Using \eqref{eq:estimate1}, \eqref{eq:estimate2} and the polynomial identity
		\begin{align*}
		\prod_{p=1}^{r}X_{p}-\prod_{p=1}^{r}Y_{p}=
		\sum_{j=1}^{r}(\prod_{p=1}^{j-1}X_{p}) (X_{j}-Y_{j}) (\prod_{p=j+1}^{r}Y_{p}),
		\end{align*}
		we obtain the following estimation:
		\begin{align*}
		&
		\left|
		\frac{1}{n^{\sum_{a=1}^{r-1}k_{a}+l}}
		\prod_{a=1}^{r-1}g_{k_{a}}\left(\frac{m_{a}}{n}\right)\,
		h_{l}\left(\frac{m_{1}+\cdots +m_{r-1}}{n}\right)-
		\prod_{a=1}^{r-1}\frac{1}{m_{a}^{k_{a}}}\,
		\frac{1}{(m_{1}+\cdots +m_{r-1})^{l}}
		\right| \\
		&\le
		\sum_{j=1}^{r-1}
		\prod_{a=1}^{j-1}\left|\frac{1}{n^{k_{a}}}g_{k_{a}}\left(\frac{m_{a}}{n}\right)\right| \,
		\left|\frac{1}{n^{k_{j}}}g_{k_{j}}\left(\frac{m_{j}}{n}\right)-\frac{1}{m_{j}^{k_{j}}} \right|
		\prod_{a=j+1}^{r-1}\frac{1}{m_{a}^{k_{a}}} \,
		\frac{1}{(m_{1}+\cdots +m_{r-1})^{l}} \\
		&\quad {}+\prod_{a=1}^{r-1}\left|\frac{1}{n^{k_{a}}}g_{k_{a}}\left(\frac{m_{a}}{n}\right)\right| \,
		\left|\frac{1}{n^{l}}h_{l}\left(\frac{m_{1}+\cdots +m_{r-1}}{n}\right)-\frac{1}{(m_{1}+\cdots +m_{r-1})^{l}} \right| \\
		&\le
		\frac{C_{\kk, l}'}{n}\sum_{j=1}^{r-1}
		\biggl\{
		\frac{1}{m_{1}^{k_{1}} \cdots m_{j}^{k_{j}-1} \cdots m_{r-1}^{k_{r-1}}} \,
		\frac{1}{(m_{1}+\cdots +m_{r-1})^{l}}  \\
		& \qquad \qquad \qquad {}+\frac{1}{m_{1}^{k_{1}} \cdots m_{r-1}^{k_{r-1}}}
		\frac{1}{(m_{1}+\cdots +m_{r-1})^{l-1}}\biggr\} \\
		&\le \frac{C_{\kk, l}'}{n}\,
		\frac{r}{m_{1} \cdots m_{r-1}}
		\end{align*}
		for some constant $C_{\kk, l}'$ which does not depend on $n$.
		Hence
		\begin{align*}
		\left|B_{n}(\kk)-\tilde{B}_{n}(\kk)\right| &\le
		\frac{C_{\kk, l}'}{n}
		\sum_{\substack{m_{1}, \ldots , m_{r-1}>0 \\ m_{1}+\cdots +m_{r-1}<n/2}}
		\frac{r}{m_{1} \cdots m_{r-1}}
		\le
		C_{\kk, l}' \, \frac{r}{n}
		\left(\sum_{n/2>m>0}\frac{1}{m}\right)^{r-1} \\
		&\le
		C_{\kk, l}' \, \frac{r}{n} \left(1+\log{\frac{n+1}{2}}\right)^{r-1} \to 0
		\end{align*}
		as $n \to \infty$.
		Therefore it holds that
		\begin{align*}
		\lim_{n \to \infty}B_{n}(\kk; l)=\lim_{n \to \infty}\tilde{B}_{n}(\kk; l)=\zeta^{MT}(\kk; l).
		\end{align*}
		This completes the proof.
	\end{proof}

	\section{Proof of Theorem \ref{thm:omega_fsmzv}}

	\subsection{$q$-multiple polylogarithm}
	In this subsection, we briefly recall the $q$-multiple polylogarithm, introduced by Zhao \cite{Zhao07}, of one variable $L_\kk(t)$ and its shuffle relations.

	In what follows, we fix a complex parameter $q$ such that $|q|<1$.
	Define $F_{k}(m)$ for positive integers $k$ and $m$ by
	\begin{align*}
	F_{k}(m)=\frac{q^{(k-1)m}}{[m]^{k}} \in \C.
	\end{align*}
	We introduce the letter $\hat{1}$ and set $\widehat{\N}=\N \sqcup \{\hat{1}\}$.
	Hereafter we call a tuple of the elements of $\widehat{\N}$ an extended index.
	In order to give a closed formula for the shuffle product for the $q$-multiple polylogarithm, we also need the term
	\begin{align*}
	F_{\hat{1}}(m)=\frac{q^{m}}{[m]}\in \C
	\end{align*}
	for $m \in \N$.
	For a non-empty extended index $\kk=(k_1,\ldots,k_r)\in \widehat{\N}^r$ we define the $q$-multiple polylogarithm of one variable $L_\kk(q;t)$ by
	\begin{equation*}\label{eq:def_mp}
	L_\kk(t)=\sum_{m_1>\cdots>m_r>0} t^{m_1} \prod_{a=1}^r F_{k_a}(m_a)
	\end{equation*}
	as a formal power series in $\C[[t]]$.
	We set $L_\emptyset(t)=1$.

	\

	Recall the algebraic setup (see \cite[\S2.2]{Takeyama19} and \cite{Zhao}).
	Let $\hbar$ be a formal variable and set $\mathcal{C}=\mathbb{Q}[\hbar, \hbar^{-1}]$.
	We denote the unital non-commutative polynomial ring over $\mathcal{C}$
	with two indeterminates $a$ and $b$ by
	\[\mathfrak{H}=\mathcal{C}\langle a,b\rangle.\]
	For $k \in \widehat{\N}$ we define $e_{k} \in \mathfrak{H}$ by
	\begin{align*}
	e_{\hat{1}}=ab, \qquad
	e_{k}=a^{k-1}(a+\hbar)b \quad (k \in \N).
	\end{align*}
	Let
	\[\widehat{\mathfrak{H}}^{1}=\mathcal{C}\langle e_{k}\mid k\in \widehat{\N}\rangle\]
	be
	the subalgebra freely generated by the set $\{e_{k}\}_{k \in \widehat{\mathbb{N}}}$.
	For a non-empty extended index $\kk=(k_{1}, \ldots , k_{r}) \in \widehat{\N}^{r}$
	we set $e_{\kk}=e_{k_{1}}\cdots e_{k_{r}}$.
	For the empty index we set $e_{\emptyset}=1$.
	Then, the set $\{e_\kk\mid \kk:\mbox{extended index}\}$ is a linear basis of $\widehat{\mathfrak{H}}^{1} $.

	We define the $\Q$-linear action of $\mathcal{C}$ on $\C[[t]]$ by $(\hbar f)(t)=(1-q)f(t)$ for $f\in \C[[t]]$ and
	then the $\mathcal{C}$-module homomorphism
	\[ \L:\widehat{\mathfrak{H}}^1\rightarrow \C[[t]],\ e_\kk\mapsto L_\kk(t).\]

	We recall that the map $\L:\widehat{\mathfrak{H}}^1\rightarrow \C[[t]]$ is viewed as an evaluation map.
	While the map $\zeta:\mathfrak{h}^0\rightarrow \R$ is given by an iterated integral of the integrands $\frac{dt}{t},\frac{dt}{1-t}$ corresponding to the indeterminates $x_0,x_1\in \mathfrak{h}$, we view the indeterminates $a,b,\hbar$ as operations on the formal power series ring $\C[[t]]$ (or the ring of holomorphic functions on $D=\{t\in \C: |t|<1\}$).
	They are defined for $f(t)=\sum_{n\ge0}c_n t^n\in \C[[t]]$ by
	\begin{align*}
	(af)(t)=(1-q)\sum_{n\ge0} f(q^n t)= (1-q)\sum_{n\ge0} c_n t^n \frac{q^n}{[n]},\\
	\quad (bf)(t)=\frac{t}{1-t}f(t),\quad (\hbar f)(t)=(1-q)f(t),
	\end{align*}
	where $a$ is only defined for $f$ being $f(0)=0$.
	By definition, $\hbar$ commutes with others.
	With this, for any extended index $\kk\in \widehat{\N}^r$ it can be shown that
	\[ L_\kk(t)=  e_\kk (1) .\]
	For example, one can compute
	\begin{align*}
	L_{\hat{1},1}(t)&=ab(a+\hbar)b(1)=ab(a+\hbar) \left( \sum_{n>0}t^n \right)= ab \left( \sum_{n>0}t^n \frac{1}{[n]}\right)\\
	&=a \left( \sum_{m>n>0}t^m \frac{1}{[n]}\right) = \sum_{m>n>0}t^m \frac{q^m}{[m][n]},
	\end{align*}
	where for the third equality we have used
	\begin{align*}
	\frac{q^n}{[n]}+(1-q)=\frac{q^n(1-q)+(1-q)(1-q^n)}{1-q^n}=\frac{1}{[n]}.
	\end{align*}
	Thus, the map $\L:\widehat{\mathfrak{H}}^1\rightarrow \C[[t]]$ is alternatively defined by $\L(e_\kk)=e_\kk(1)$.

	We define the shuffle product $\qsh$ as the $\mathcal{C}$-bilinear map
	$\qsh: \mathfrak{H} \times \mathfrak{H}\to \mathfrak{H}$ by
	\begin{align*}
	&
	a w\qsh a w'=a (a w \qsh w'+w\qsh a w'+\hbar\, w\qsh w'), \\
	&
	b w\qsh w'=w\qsh b w'=b (w \qsh w')
	\end{align*}
	for $w, w' \in \mathfrak{H}$, with the initial condition $1\qsh w=w \qsh 1=w$.
	For example, $e_1\qsh e_1=2abab+ab\hbar b+2\hbar bab + \hbar b\hbar b= e_1e_1+e_1e_{\hat{1}}$.
	The element $w \qsh w'$ for $w, w' \in \widehat{\mathfrak{H}}^{1}$ lies in $\widehat{\mathfrak{H}}$, but, as we see in the example, it turns out to be a $\Q$-linear combination of words in $\widehat{\mathfrak{H}}^1$:

	\begin{lem}\label{lem:shuffle-rel}
		For $w, w' \in \widehat{\mathfrak{H}}^{1}$ we have $w \qsh w'\in \widehat{\mathfrak{H}}^{1}$.
	\end{lem}
	\begin{proof}
		For the proof, see Proposition 12.2.20 in \cite{Zhao}.
	\end{proof}

	With Lemma \ref{lem:shuffle-rel}, we can state the shuffle relation for $q$-multiple polylogarithms.
	\begin{prop}\label{prop:shuffle-rel}
		For $w, w' \in \widehat{\mathfrak{H}}^{1}$ it holds that $\L(w \qsh w')=\L(w)\L(w')$.
	\end{prop}
	\begin{proof}
		By Lemma \ref{lem:shuffle-rel}, the expression $\L(w \qsh w')$ is well-defined.
		Then the desired result follows from easily checked identities $(af)\cdot (ag)=a(f\cdot ag)+a(af \cdot g)+\hbar a (f\cdot g)$ for $f,g\in \C[[t]]$ with $f(0)=g(0)=0$ and $(bf) \cdot g=b(f\cdot g)=f\cdot (bg)$ for $f,g\in \C[[t]]$, where $f\cdot g$ means the product on $\C[[t]]$.
	\end{proof}

	\subsection{Multiple harmonic $q$-series at primitive roots of unity}\label{sec:previous_result}

	In this subsection, we recall the main results of our previous work \cite{BTT18}.

	For a word $w\in \widehat{\mathfrak{H}}^1$, we write $u_m(w)$ the coefficient of $t^m$ in $\L(w)\in \C[[t]]$
	\begin{equation}\label{eq:def_u}
	\sum_{n\ge0} u_m (w) t^n := \L(w)
	\end{equation}
	and define
	\begin{equation}\label{eq:def_zn}
	z_n (w): = \sum_{m=1}^{n-1} u_m(w).
	\end{equation}
	These are viewed as $\mathcal{C}$-module homomorphisms from $\widehat{\mathfrak{H}}^1$ to $\C$ with the $\mathcal{C}$-linear action defined by $\hbar \alpha=(1-q)\alpha$ for $\alpha\in \C$.
	Namely, we have $z_n (\hbar e_\kk)=(1-q) z_n(e_\kk)$ for an extended index $\kk=(k_1,\ldots,k_r) \in \widehat{\N}^r$, where by definition
	\begin{equation*}\label{eq:prop_zn}
	z_n (e_\kk) = \sum_{n>m_1>\cdots>m_r>0}  \prod_{a=1}^r F_{k_a}(m_a).
	\end{equation*}

	In \cite{BTT18}, we studied the value $z_n (e_\kk)$ for an index $\kk$ (not extended index) at $q=\zeta_n$ a primitive $n$-th root of unity.
	Note that the substitution $q=\zeta_n$ does not make sense if $\zeta_n$ is not primitive.
	For an index $\kk=(k_1,\ldots,k_r)$, we write
	\begin{equation}\label{eq:def_znq}
	z_n(\kk ;q):=z_n(e_\kk) = \sum_{n>m_1>\cdots>m_r>0} \prod_{a=1}^r \frac{q^{(k_a-1)m_a}}{[m_a]^{k_a}}.
	\end{equation}

	\begin{thm}\cite[Theorems 1.1 and 1.2]{BTT18} \label{thm:BTT}
		\begin{enumerate}
			\item For any index $\kk\in \N^r$ we have
			\[ \left(z_p(\kk;\zeta_p) \mod (1-\zeta_p)\Z[\zeta_p]\right)_p = \zeta_{\mathcal{A}}(\kk) \in \mathcal{A}.\]
			\item For any index $\kk=(k_{1}, \ldots , k_{r})\in \N^r$ the limit
			\begin{align*}
			\xi(\kk)=\lim_{n \to \infty}z_{n}(\kk; e^{2\pi i/n})
			\end{align*}
			exists, and it holds that $\mathrm{Im}\, \xi(\kk) \in \pi \mathcal{Z}$ and
			\begin{align*}
			\mathrm{Re}\,\xi(\kk)\equiv \zeta_{\mathcal{S}}(\kk) \quad \mathrm{mod} \quad
			\zeta(2)\mathcal{Z}.
			\end{align*}
		\end{enumerate}
	\end{thm}

	\subsection{Proof of \eqref{eq:omega_fmzv}}

	We first prove the following theorem and then \eqref{eq:omega_fmzv}.

	\begin{thm}\label{thm:q-Kamano}
		For any index $\kk=(k_{1}, \ldots , k_{r}) \in \N^{r}$ with $r \ge 2$ and primitive $n$-th root of unity $\zeta_n$, it holds that
		\begin{align*}
		\omega_{n}(\kk; \zeta_{n})=(-1)^{k_{r}}
		\sum_{j=1}^{k_{r}}\binom{k_{r}-1}{j-1}(1-q)^{k_{r}-j}
		z_{n}(a^{j}(e_{k_{1}} \qsh \cdots \qsh e_{k_{r-1}}))\big|_{q=\zeta_n}.
		\end{align*}
	\end{thm}
	\begin{proof}
		It is easily seen that the value $\omega_{n}(\kk; q)$ is equal to
		the coefficient of $t^{n}$ in the product $\prod_{a=1}^{r}L_{k_{a}}(t)$ of $q$-single polylogarithms.
		From Proposition \ref{prop:shuffle-rel} it holds that
		\begin{align*}
		\prod_{a=1}^{r}L_{k_{a}}(t)=
		\L(e_{k_{r}})\L(e_{k_{1}} \qsh \cdots \qsh e_{k_{r-1}}).
		\end{align*}
		Hence by definition \eqref{eq:def_u} we find that
		\begin{align*}
		\omega_{n}(\kk;q)=\sum_{m=1}^{n-1}u_{n-m}(e_{k_{r}})
		u_{m}(e_{k_{1}} \qsh \cdots \qsh e_{k_{r-1}}).
		\end{align*}
		For $1 \le m <n$ and $k \in \N$, we compute
		\begin{align*}
		u_{n-m}(e_{k})\big|_{q= \zeta_{n}}&=F_{k}(n-m)\big|_{q= \zeta_{n}}=(-1)^{k}\frac{q^{m}}{[m]^{k}} \bigg|_{q=\zeta_n}\\
		&=(-1)^{k}\sum_{j=1}^{k}\binom{k-1}{j-1}(1-q)^{k-j}
		\frac{q^{jm}}{[m]}\bigg|_{q=\zeta_n}.
		\end{align*}
		Note that the element $e_{k_{1}} \qsh \cdots \qsh e_{k_{r-1}}$ belongs to the $\mathcal{C}$-submodule
		$\sum_{l \in \N}e_{l}\widehat{\mathfrak{H}}^{1}$.
		Using the equality
		\begin{align*}
		\frac{q^{jm}}{[m]^{j}}F_{l}(m)\big|_{q=\zeta_n} =F_{l+j}(m)\big|_{q=\zeta_n}
		\end{align*}
		and $a^{j}e_{l}=e_{l+j} \, (j, l \in \N)$, we see that
		\begin{align*}
		\frac{q^{jm}}{[m]^{j}}u_{m}(e_{k_{1}} \qsh \cdots \qsh e_{k_{r-1}})\big|_{q= \zeta_{n}}=
		u_{m}(a^{j}(e_{k_{1}} \qsh \cdots \qsh e_{k_{r-1}}))\big|_{q= \zeta_{n}}.
		\end{align*}
		Therefore by definition \eqref{eq:def_zn} we obtain
		\begin{align*}
		\omega_{n}(\kk; \zeta_{n})&=(-1)^{k_{r}}\sum_{j=1}^{k_{r}}\binom{k_{r}-1}{j-1}(1-q)^{k_{r}-j}
		\sum_{m=1}^{n-1}u_{m}(a^{j}(e_{k_{1}} \qsh \cdots \qsh e_{k_{r-1}}))\big|_{q= \zeta_{n}} \\
		&=(-1)^{k_{r}}\sum_{j=1}^{k_{r}}\binom{k_{r}-1}{j-1}(1-q)^{k_{r}-j}
		z_{n}(a^{j}(e_{k_{1}} \qsh \cdots \qsh e_{k_{r-1}}))\big|_{q=\zeta_n}.
		\end{align*}
		This completes the proof.
	\end{proof}

	We are now in a position to prove \eqref{eq:omega_fmzv}.

	\begin{proof}[Proof of \eqref{eq:omega_fmzv}]
		When $p$ is prime and $q=\zeta_{p}$, for $m<p$
		it holds that $F_{\hat{1}}(m)\equiv m^{-1} \mod{(1-\zeta_p)\Z[\zeta_p]}$ and $F_{k}(m)\equiv m^{-k} \mod{(1-\zeta_p)\Z[\zeta_p]}$ for any $k \in \N$.
		Hence, under the identification $\Z[\zeta_p]/(1-\zeta_p)\Z[\zeta_p] =\F_p$, for any $w \in \widehat{\mathfrak{H}}^{1}$ it holds that
		\begin{align*}
		( z_{p}(w)\big|_{q=\zeta_{p}} \mod{(1-\zeta_p)\Z[\zeta_p]})_{p}=
		\zeta_{\mathcal{A}}(\rho(w)),
		\end{align*}
		where $\rho: \, \widehat{\mathfrak{H}}^{1} \to \mathfrak{h}^{1}$ is the $\Q$-algebra homomorphism given by
		\begin{align*}
		\rho(\hbar)=0, \quad
		\rho(e_{\hat{1}})=y_{1}, \quad
		\rho(e_{k})=y_{k} \,\, (k \in \N).
		\end{align*}
		Note that $\widehat{\mathfrak{H}}^{1}$ (resp. $\mathfrak{h}^{1}$) is invariant under
		the left multiplication with $a$ (resp. $x_0$) because
		$ae_{\hat{1}}=e_{2}-\hbar e_{\hat{1}}, \, ae_{k}=e_{k+1}$ and $x_0 y_{k}=y_{k+1}$ for $k \in \N$.
		These relations imply that
		\begin{align}
		\rho(a e_{\kk})=x_0\rho(e_{\kk})
		\label{eq:left-mul-rho}
		\end{align}
		for any non-empty extended index $\kk \in \widehat{\N}^{r}$.
		Hence, by Theorem \ref{thm:q-Kamano}, for any index $\kk=(k_1,\ldots,k_r)\in \N^r$ we obtain
\begin{align*}
(\omega(\kk;\zeta_p) \mod (1-\zeta_p )\Z[\zeta_p] )_{p}=(-1)^{k_r}
\zeta_{\mathcal{A}}(x_0^{k_r}\rho(e_{k_{1}} \qsh \cdots \qsh e_{k_{r-1}})).
\end{align*}
		Because of Theorem \ref{thm:BTT} (i), it suffices to show that
		\[\rho(w_{1} \qsh w_{2})=\rho(w_{1}) \shuffle \rho(w_{2})\]
		for any $w_{1}, w_{2} \in \widehat{\mathfrak{H}}^{1}$.
		This is done by induction on the sum of weights (degrees) of $w_{1}$ and $w_{2}$.
		We may assume that $w_{1}$ and $w_{2}$ are monomials in $\{e_{k}\}_{k \in \widehat{\N}}$.
		If $w_{1}=1$ or $w_{2}=1$, it is trivial.
		Hence we consider the case where
		$w_{j}=e_{k_{j}}w_{j}' \, (j=1, 2)$ for some $k_{j}' \in \widehat{\N}$.

		First we consider the case $k_{1}=k_{2}=\hat{1}$.
		Then
		\begin{align*}
		w_{1}\qsh w_{2}&=abw_{1}'\qsh abw_{2}'=
		a(bw_{1}'\qsh abw_{2}'+abw_{1}'\qsh bw_{2}'+\hbar bw_{1}'\qsh bw_{2}') \\
		&=e_{\hat{1}}(w_{1}'\qsh e_{\hat{1}}w_{2}'+e_{\hat{1}}w_{1}'\qsh w_{2}')+
		e_{\hat{1}}(e_{1}-e_{\hat{1}})(w_{1}'\qsh w_{2}').
		\end{align*}
		Using the induction hypothesis and $\rho(e_{\hat{1}})=\rho(e_1)=y_1=x_1$ we see that
		\begin{align*}
		\rho(w_{1}\qsh w_{2})&=x_1(\rho(w_{1}')\shuffle x_1\rho(w_{2}')+x_1\rho(w_{1}')\shuffle \rho(w_{2}')) \\
		&=x_1\rho(w_{1}')\shuffle x_1\rho(w_{2}')=\rho(e_{\hat{1}}w_{1}')\shuffle \rho(e_{\hat{1}}w_{2}')=\rho(w_{1})\shuffle \rho(w_{2}).
		\end{align*}
		For the case $(k_{1}, k_{2})=(1, \hat{1})$ or $(1, 1)$
		we obtain the desired equality in a similar way by using
		\begin{align*}
		&
		(e_{1}w_{1}')\qsh (e_{1}w_{2}')=e_{1}(w_{1}'\qsh e_{\hat{1}}w_{2}'+e_{\hat{1}}w_{1}'\qsh w_{2}')+
		e_{1}(e_{1}-e_{\hat{1}})(w_{1}'\qsh w_{2}'), \\
		&
		(e_{1}w_{1}')\qsh (e_{\hat{1}}w_{2}')=e_{1}(w_{1}'\qsh e_{\hat{1}}w_{2}')+
		e_{\hat{1}}(e_{\hat{1}}w_{1}'\qsh w_{2}')+e_{\hat{1}}(e_{1}-e_{\hat{1}})(w_{1}'\qsh w_{2}').
		\end{align*}

		Next we consider the case where $k_{1}=\hat{1}$ and $k_{2} \ge 2$.
		In this case, one can write $w_{2}=a w_{2}''$ with $w_{2}'' \in \widehat{\mathfrak{H}}^{1}$.
		Hence we see that
		\begin{align*}
		w_{1} \qsh w_{2}&=abw_{1}' \qsh aw_{2}''=
		e_{\hat{1}}(w_{1}' \qsh w_{2})+a(w_{1} \qsh w_{2}'')+\hbar e_{\hat{1}}(w_{1}'\qsh w_{2}'').
		\end{align*}
		Hence the induction hypothesis and \eqref{eq:left-mul-rho} imply that
		\begin{align*}
		\rho(w_{1} \qsh w_{2})&=x_1 (\rho(w_{1}') \shuffle  \rho(w_{2}))+x_0 (\rho(w_{1}) \shuffle \rho(w_{2}'')) \\
		&=x_1 (\rho(w_{1}') \shuffle  x_0 \rho(w_{2}''))+x_0 (x_1 \rho(w_{1}') \shuffle \rho(w_{2}'')) \\
		&= x_1\rho(w_{1}') \shuffle  x_0 \rho(w_{2}'')=\rho(w_{1}) \shuffle \rho(w_{2}).
		\end{align*}

		Finally we consider the case where $k_{j} \ge 2 \, (j=1, 2)$.
		We write $w_{j}=a w_{j}''$ with $w_{j}'' \in \widehat{\mathfrak{H}}^{1}$.
		Then it holds that
		\begin{align*}
		w_{1} \qsh w_{2}=a(w_{1}'' \qsh w_{2}+w_{1} \qsh w_{2}''+\hbar \, w_{1}''\qsh w_{2}'').
		\end{align*}
		Using the induction hypothesis and \eqref{eq:left-mul-rho}, we see that
		\begin{align*}
		\rho(w_{1} \qsh w_{2})&=x_0 (\rho(w_{1}'') \shuffle x_0 \rho(w_{2}'')+x \rho(w_{1}'') \shuffle \rho(w_{2}''))=
		x_0 \rho(w_{1}'') \shuffle x_0 \rho(w_{2}'')\\
		&=\rho(w_{1})\shuffle \rho(w_{2}).
		\end{align*}
		This completes the proof.
	\end{proof}

	\begin{rem}\label{rem:Kamano}
		In \cite{Kamano}, Kamano introduced the finite Mordell-Tornheim multiple zeta value for an index $\kk=(k_1,\ldots,k_r)\in \N^r$ by
		\[ \zeta^{MT}_{\mathcal{A}}(\kk) = \left( \sum_{\substack{m_1,\ldots,m_{r-1}\ge1\\m_1+\cdots+m_{r-1}<p}} \frac{1}{m_1^{k_1}\cdots m_{r-1}^{k_{r-1}} (m_1+\cdots+m_{r-1})^{k_r}}\mod{p} \right)_p \in \mathcal{A}\]
		and, as a finite analogue of \eqref{eq:MT_mz}, proved that
		\begin{equation}\label{eq:Kamano} \zeta^{MT}_{\mathcal{A}}(\kk) = \zeta_{\mathcal{A}}(x_0^{k_r}(y_{k_1}\shuffle \cdots \shuffle y_{k_{r-1}})).
		\end{equation}
		Since $\omega_{\mathcal{A}}(\kk)=(-1)^{k_r} \zeta_{\mathcal{A}}^{MT}(\kk)$ for each index $\kk=(k_1,\ldots,k_r)\in \N^r$, Theorem \ref{thm:q-Kamano} can be viewed as a lift of Kamano's result \eqref{eq:Kamano} to the values $\omega_n(\kk;\zeta_n)$.
	\end{rem}

	\subsection{Proof of \eqref{eq:omega_smzv}}
	The proof of \eqref{eq:omega_smzv} is based on identities of words.
	Let $\phi: \mathfrak{h}^{1} \to \R$ be the $\Q$-linear map given by
	$\phi(1)=1$ and
	\begin{align}
	\phi(y_{k_{1}} \cdots y_{k_{r}})=\sum_{a=0}^{r}(-1)^{k_{1}+\cdots +k_{a}}
	(y_{k_{a}}y_{k_{a-1}}\cdots y_{k_{1}}) \shuffle
	(y_{k_{a+1}}y_{k_{a+2}}\cdots y_{k_{r}}).
	\label{eq:def-phi}
	\end{align}

	\begin{thm}\label{thm:identity_words}
		For any index $(k_{1}, \ldots , k_{r}) \in \N^{r}$ with $r \ge 2$, we have
		\begin{align*}
		&\phi((-1)^{k_{r}}x_0^{k_{r}}(y_{k_{1}} \shuffle \cdots \shuffle y_{k_{r-1}}))\\
		&=\sum_{a=1}^{r}(-1)^{k_{a}}x_0^{k_{a}}(y_{k_{1}} \shuffle \cdots \shuffle y_{k_{a-1}}\shuffle y_{k_{a+1}} \shuffle \cdots \shuffle y_{k_{r}}).
		\end{align*}
	\end{thm}

	First, we give a proof of \eqref{eq:omega_smzv}, and then prove Theorem \ref{thm:identity_words} below.

	\begin{proof}[Proof of \eqref{eq:omega_smzv}]
		Since $\zeta^{\shuffle}: \mathfrak{h}^{1}_{\shuffle} \to \R$ is a homomorphism,
		it follows from the definition \eqref{eq:def_zeta_s} that $\zeta_{\mathcal{S}}^{\shuffle}(w)=\zeta^{\shuffle}(\phi(w))$
		for any $w \in \mathfrak{h}^{1}$.
		By Theorem \ref{thm:identity_words} and \eqref{eq:MT_mz}, we see that
		\begin{align*}
		&(-1)^{k_{r}}\zeta_{\mathcal{S}}^{\shuffle}\left(x_0^{k_{r}}(y_{k_{1}} \shuffle \cdots \shuffle y_{k_{r-1}})\right)\\
		&=\zeta^{\shuffle}\left(\phi\left((-1)^{k_{r}}x_0^{k_{r}}(y_{k_{1}} \shuffle \cdots \shuffle y_{k_{r-1}})\right)\right)\\
		&=\sum_{a=1}^{r}(-1)^{k_{a}} \zeta^\shuffle \left(x_0^{k_{a}}(y_{k_{1}} \shuffle \cdots \shuffle y_{k_{a-1}}\shuffle y_{k_{a+1}} \shuffle \cdots \shuffle y_{k_{r}})\right)\\
		&=\sum_{a=1}^{r}(-1)^{k_{a}}\zeta^{MT}(k_1,\ldots,k_{a-1},k_{a+1},\ldots,k_r;k_a),
		\end{align*}
		which by Theorem \ref{thm:SMZV} is equal to $\Omega(k_1,\ldots,k_r)$.
		Taking modulo $\zeta(2)$, we get the desired result.
	\end{proof}

	In what follows, we will prove Theorem \ref{thm:identity_words} by calculating the generating series
	of both sides.
	For that purpose we introduce the element
	\begin{align*}
	y(t)=\sum_{k=1}^{\infty}y_{k}t^{k-1}=\frac{1}{1-x_0t}\,x_1 \, \in \, \mathfrak{h}^{1}[[t]]
	\end{align*}
	with the indeterminate $t$.
	We begin with the following lemma, which is used to prove Propositions \ref{prop:zt-sh-formula1}, \ref{prop:zt-phi-formula2} and \ref{prop:zt-phi-formula1} below, and then, give a proof of Theorem \ref{thm:identity_words}

	\begin{lem}\label{lem:zt-sh}
		For $w, w' \in \mathfrak{h}^{1}$ it holds that
		\begin{align*}
		y(t_{1}) w \shuffle y(t_{2}) w'=y(t_{1}+t_{2})
		\left(w\shuffle y(t_{2})w'+y(t_{1})w \shuffle w' \right).
		\end{align*}
	\end{lem}

	\begin{proof}
		We denote the left side by $I$.
		Noting $y(t)=x_1+tx_0y(t)$, one computes
		\begin{align*}
		I&=\left(x_1+t_1x_0y(t_1)\right)w \shuffle \left(x_1+t_2x_0y(t_2)\right)w' \\
		&=x_1w \shuffle  x_1w' + t_1x_0y(t_1)w\shuffle x_1w'+ x_1w \shuffle   t_2x_0y(t_2)w' + t_1x_0y(t_1)w\shuffle t_2x_0y(t_2)w' \\
		&= x_1(w\shuffle x_1w')+x_1(x_1w\shuffle w') \\
		&+t_1x_0 (y(t_1)w \shuffle x_1 w')+  t_1 x_1(x_0y(t_1)w\shuffle w')\\
		&+t_2x_1(w\shuffle x_0 y(t_2)w') + t_2 x_0 (x_1w\shuffle y(t_2)w')\\
		&+t_1t_2x_0(y(t_1)w\shuffle x_0 y(t_2)w' )+t_1t_2x_0( x_0y(t_1)w\shuffle y(t_2)w')\\
		&=x_0\left\{t_1y(t_1)w\shuffle x_1 w' + x_1w\shuffle t_2 y(t_2) w'+ t_1y(t_1)w\shuffle x_0t_2 y(t_2)w'+ t_1x_0y(t_1)w\shuffle t_2y(t_2)w'\right\}\\
		&+x_1\left\{ w\shuffle x_1w'+x_1w\shuffle w' + x_0t_1y(t_1)w\shuffle w'+ w\shuffle x_0 t_2y(t_2)w'\right\}\\
		&= x_0 \left( t_1y(t_1)w\shuffle y(t_2) w' +t_2 y(t_1)w\shuffle y(t_2)w' \right) \\
		&+x_1\left(
		w  \shuffle y(t_{2})w' +y(t_{1})w \shuffle w'  \right)\\
		&=(t_{1}+t_{2})x_0I+  x_1\left(
		w  \shuffle y(t_{2})w' +y(t_{1})w \shuffle w'  \right).
		\end{align*}
		Hence, $(1-(t_1+t_2)x_0)I=x_1 \left(
		w  \shuffle y(t_{2})w' +y(t_{1})w \shuffle w'  \right)$, which implies the desired formula.
	\end{proof}

	From Lemma \ref{lem:zt-sh} one can compute
	\[y(t_1)\shuffle y(t_2) = y(t_{1}+t_{2})
	\left( y(t_{2})+y(t_{1}) \right)\]
	and
	\begin{align*}
	(y(t_1)\shuffle y(t_2))\shuffle y(t_3)&=y(t_{1}+t_{2})
	\left( y(t_{2})+y(t_{1}) \right)\shuffle y(t_3)\\
	&=y(t_{1}+t_{2}) y(t_1) \shuffle y(t_3) + y(t_{1}+t_{2}) y(t_2) \shuffle y(t_3) \\
	&=y(t_1+t_2+t_3)(y(t_1)\shuffle y(t_3)+y(t_1+t_2)y(t_1))\\
	&+y(t_1+t_2+t_3)(y(t_2)\shuffle y(t_3)+y(t_1+t_2)y(t_2))\\
	&=y(t_1+t_2+t_3)\left( y(t_1+t_3)(y(t_1)+y(t_3))+y(t_1+t_2)y(t_1)\right)\\
	&+y(t_1+t_2+t_3)\left( y(t_2+t_3)(y(t_2)+y(t_3))+y(t_1+t_2)y(t_2)\right)\\
	&=y(t_1+t_2+t_3)y(t_1+t_2)y(t_1) + y(t_1+t_2+t_3)y(t_1+t_3)y(t_1) \\
	&+y(t_1+t_2+t_3)y(t_1+t_2)y(t_2) + y(t_1+t_2+t_3)y(t_2+t_3)y(t_2)\\
	&+y(t_1+t_2+t_3)y(t_2+t_3)y(t_3) + y(t_1+t_2+t_3)y(t_1+t_3)y(t_3).
	\end{align*}
	Using Lemma \ref{lem:zt-sh} repeatedly, we obtain the following formula
	for the shuffle product of $y(t)$:

	\begin{prop}\label{prop:zt-sh-formula1}
		It holds that
		\begin{align*}
		y(t_{1}) \shuffle \cdots \shuffle y(t_{r})=\sum_{\sigma \in \mathfrak{S}_{r}}
		\prod_{1\le j \le r}^{\curvearrowleft}y\left({\textstyle \sum_{i=1}^{j}t_{\sigma(i)}} \right),
		\end{align*}
		where $\prod_{m \le j \le n}^{\curvearrowleft}X_{j}$ denotes
		the ordered product $X_{n}X_{n-1}\cdots X_{m}$ for $m \le n$.
	\end{prop}

	We denote by $[r]$ the set $\{1, 2, \ldots , r\}$ for $r \ge 1$.
	For a subset $I=\{p_{1}, \ldots , p_{s}\}$ of $[r]$ with $|I|=s$,
	we define
	\begin{align*}
	S_{I}(t_{1}, \ldots , t_{r})&=y(t_{p_{1}}) \shuffle \cdots \shuffle y(t_{p_{s}})\\
	&=\sum_{\sigma\in \mathfrak{S}_{s}}\prod_{1\le j\le s}^{\curvearrowleft}y\left({\textstyle \sum_{i=1}^{j}t_{p_{\sigma(i)}}} \right).
	\end{align*}

	\begin{prop}\label{prop:zt-phi-formula2}
		For any subset $I=\{p_{1}, \ldots , p_{s}\}$ of $[r]$ with $|I|=s$, it holds that
		\begin{align*}
		\phi(y(u)S_{I}(t_{1}, \ldots , t_{r}))&=
		y(u)S_{I}(t_{1}, \ldots , t_{r})-y(-u)\shuffle S_{I}(t_{1}, \ldots , t_{r}) \\
		&+\sum_{b\in I}y(-t_{b})
		\left(y(-u)\shuffle S_{I\setminus\{b\}}(t_{1}, \ldots , t_{r}) \right).
		\end{align*}
	\end{prop}
	\begin{proof}
		{}From the definition \eqref{eq:def-phi} of the map $\phi$, we see that
		\begin{align*}
		\phi(y(s_{1})\cdots y(s_{r}))&=\sum_{k_1,\ldots,k_r\ge1} \phi(y_{k_1}\cdots y_{k_r})s_1^{k_1-1}\cdots s_r^{k_r-1}\\
		&=\sum_{a=0}^r (-1)^{k_1+\cdots+k_a} \sum_{k_1,\ldots,k_r\ge1} y_{k_a}\cdots y_{k_1}\shuffle y_{k_{a+1}}\cdots y_{k_r} s_1^{k_1-1}\cdots s_r^{k_r-1}\\
		&=\sum_{a=0}^r(-1)^a \underbrace{y(-s_a)\cdots y(-s_1)}_a\shuffle \underbrace{y(s_{a+1})\cdots y(s_r)}_{r-a}\\
		&=\sum_{a=0}^r(-1)^a \left(\prod_{1 \le j \le a}^{\curvearrowleft} y(-s_j) \right)  \shuffle
		\left(\prod_{a+1 \le j \le r}^{\curvearrowright} y(s_j) \right).
		\end{align*}
		For simplicity we set $\alpha_{j}^{\sigma}=\sum_{i=1}^{j}t_{p_{\sigma(i)}}$
		for $1 \le j \le s$ and $\sigma \in \mathfrak{S}_{s}$.
		Proposition \ref{prop:zt-sh-formula1} and the above formula imply that
		\begin{equation}\label{eq:step1}
		\begin{aligned}
		&\phi(y(u)S_{I}(t_{1}, \ldots , t_{r})) \\
		&=\sum_{\sigma\in \mathfrak{S}_s}
		 \phi (y(u)y(\alpha_s^{\sigma})\cdots y(\alpha_1^{\sigma}))\\
		&=\sum_{\sigma\in \mathfrak{S}_s} y(u)y(\alpha_s^{\sigma})\cdots y(\alpha_1^{\sigma}) \\
		&+(-1)^1
		 \sum_{\sigma\in \mathfrak{S}_s} y(-u)\shuffle y(\alpha_s^{\sigma})\cdots y(\alpha_1^{\sigma})\\
		&+\sum_{a=2}^{s+1}(-1)^a \sum_{\sigma\in \mathfrak{S}_s} \left( \left\{ \prod_{s-a+2 \le j\le s}^{\curvearrowright}y(-\alpha_{j}^\sigma)\right\} y(-u) \right) \shuffle \left( \prod_{1 \le j\le s-a+1}^{\curvearrowleft}y(\alpha_{j}^\sigma)\right) .
		\end{aligned}
		\end{equation}
		Applying Lemma \ref{lem:zt-sh} except for the first two terms and for the last term with $a=s+1$, we obtain
		\begin{align*}
		&=y(u)S_{I}(t_1,\ldots,t_r)- y(-u)\shuffle S_{I}(t_1,\ldots,t_r)\\
		&+\sum_{\sigma\in \mathfrak{S}_s} \sum_{a=2}^{s} (-1)^a y(-t_{p_{\sigma(s-a+2)}}) \left( \left\{ \prod_{s-a+2 \le j\le s}^{\curvearrowright}y(-\alpha_{j}^\sigma)\right\} y(-u) \right) \shuffle \left( \prod_{1 \le j\le s-a}^{\curvearrowleft}y(\alpha_{j}^\sigma)\right)\\
		&+\sum_{\sigma\in \mathfrak{S}_s} \sum_{a=2}^{s} (-1)^a y(-t_{p_{\sigma(s-a+2)}}) \left( \left\{ \prod_{s-a+3 \le j\le s}^{\curvearrowright}y(-\alpha_{j}^\sigma)\right\} y(-u) \right) \shuffle \left( \prod_{1 \le j\le s-a+1}^{\curvearrowleft}y(\alpha_{j}^\sigma)\right)\\
		&+(-1)^{s+1} y(-\alpha_1^{\sigma})\cdots y(-\alpha_s^\sigma)y(-u)
		\end{align*}
		In the third term with $3 \le a \le s$,
		change $a\to a+1$ and $\sigma \to \sigma'=\sigma (s-a+2,s-a+1)$,
		where $(s-a+2,s-a+1) \in \mathfrak{S}_{s}$ is the transposition.
		Then, since $\alpha_j^{\sigma'}=\alpha_j^\sigma \ (j\neq s-a+1)$, it cancels with the second term except for the term with $a=s$.
		We see that the second term with $a=s$ also cancels with the last term by changing $\sigma \to \sigma (1, 2)$.
		For the third term with $a=2$, decomposing the range $\mathfrak{S}_{s}$ as
		$\sqcup_{l=1}^{s}\mathfrak{S}_{s}^{l}$,
		where $\mathfrak{S}_{s}^{l}=\{\sigma \in \mathfrak{S}_{s} \, | \, \sigma(s)=l \}$, we get
		\begin{align*}
		&=y(u)S_{I}(t_1,\ldots,t_r)- y(-u)\shuffle S_{I}(t_1,\ldots,t_r)\\
		&+\sum_{l=1}^s y(-t_{p_l})(y(-u)\shuffle\sum_{\sigma\in \mathfrak{S}_s^l} y(\alpha_{s-1}^\sigma)\cdots y(\alpha_1^\sigma))\\
		&=y(u)S_{I}(t_1,\ldots,t_r)- y(-u)\shuffle S_{I}(t_1,\ldots,t_r)\\
		&+\sum_{b\in I} y(-t_b) ( y(-u) \shuffle S_{I\setminus\{b\}}(t_1,\ldots,t_r)),
		\end{align*}
		which completes the proof.
	\end{proof}

	In the same way as above, we obtain the following formula.
	\begin{prop}\label{prop:zt-phi-formula1}
		For any subset $I$ of $[r]$, it holds that
		\begin{align*}
		\phi(S_{I}(t_{1}, \ldots , t_{r}))&=
		S_{I}(t_{1}, \ldots , t_{r})-\sum_{b\in I}y(-t_{b})
		S_{I\setminus\{b\}}(t_{1}, \ldots , t_{r}) .
		\end{align*}
	\end{prop}
	\begin{proof}
		We use the same notation with the proof of Proposition \ref{prop:zt-phi-formula2}.
		Replacing $y(u)$ with 1 in \eqref{eq:step1} and doing the same game, we get
		\begin{align*}
		&\phi(S_{I}(t_{1}, \ldots , t_{r}))\\
		&=\sum_{\sigma \in \mathfrak{S}_{s}}
		\sum_{a=0}^{s}(-1)^{a}
		\left(\prod_{s-a+1 \le j \le s}^{\curvearrowright} y(-\alpha_{j}^{\sigma}) \right)  \shuffle
		\left(\prod_{1 \le j \le s-a}^{\curvearrowleft} y(\alpha_{j}^{\sigma}) \right)\\
		&= \sum_{\sigma \in \mathfrak{S}_{s}}
		\prod_{1 \le j \le s}^{\curvearrowleft} y(\alpha_{j}^{\sigma}) \\
		&+\sum_{\sigma \in \mathfrak{S}_{s}}\sum_{a=1}^{s-1}(-1)^{a}y(-t_{p_{\sigma(s-a+1)}})
		\left(\prod_{s-a+2 \le j \le s}^{\curvearrowright} y(-\alpha_{j}^{\sigma}) \right)  \shuffle
		\left(\prod_{1 \le j \le s-a}^{\curvearrowleft} y(\alpha_{j}^{\sigma}) \right) \\
		&+\sum_{\sigma \in \mathfrak{S}_{s}}\sum_{a=1}^{s-1}(-1)^{a}y(-t_{p_{\sigma(s-a+1)}})
		\left(\prod_{s-a+1 \le j \le s}^{\curvearrowright} y(-\alpha_{j}^{\sigma}) \right)  \shuffle
		\left(\prod_{1 \le j \le s-a-1}^{\curvearrowleft} y(\alpha_{j}^{\sigma}) \right) \\
		&+\sum_{\sigma \in \mathfrak{S}_{s}}(-1)^{s}
		\prod_{1 \le j \le s}^{\curvearrowright} y(-\alpha_{j}^{\sigma})\\
		&=\sum_{\sigma \in \mathfrak{S}_{s}}
		\prod_{1 \le j \le s}^{\curvearrowleft} y(\alpha_{j}^{\sigma})-
		\sum_{\sigma \in \mathfrak{S}_{s}}y(-t_{p_{\sigma(s)}})
		\prod_{1 \le j \le s-1}^{\curvearrowleft} y(\alpha_{j}^{\sigma})\\
		&=
		S_{I}(t_{1}, \ldots , t_{r})-\sum_{b\in I}y(-t_{b})
		S_{I\setminus\{b\}}(t_{1}, \ldots , t_{r}).
		\end{align*}
		We complete the proof.
	\end{proof}

	We are ready to prove Theorem \ref{thm:identity_words}.

	\begin{proof}[Proof of Theorem \ref{thm:identity_words}]
		Let
		\[
		J(t_1,\ldots,t_{r-1},t_r)
		:=\sum_{k_1,\ldots,k_r\ge1} (-1)^{k_r}x_0^{k_r} ( y_{k_1}\shuffle \cdots \shuffle y_{k_{r-1}})
		t_1^{k_1-1}\cdots t_{r-1}^{k_{r-1}-1}t_r^{k_r-1}.
		\]
		Then, we have
		\begin{align*}
		&
		\sum_{k_1,\ldots,k_r\ge1} \sum_{a=1}^{r}(-1)^{k_{a}}x_0^{k_{a}}
		(y_{k_{1}} \shuffle \cdots \shuffle y_{k_{a-1}}\shuffle y_{k_{a+1}} \shuffle \cdots \shuffle y_{k_{r}})
		t_1^{k_1-1}\cdots t_r^{k_r-1}\\
		&=\sum_{a=1}^r
		J(\underbrace{t_1,\ldots,t_{a-1}}_{a-1},\underbrace{t_{a+1},\ldots,t_r}_{r-a},t_a)
		\end{align*}
		Therefore, for $r\ge2$, it suffices to show that
		\begin{equation}\label{eq:gen_identity}
		\phi \big( J(t_1,\ldots,t_r)\big)=
		\sum_{a=1}^{r}
		J(\underbrace{t_1,\ldots,t_{a-1}}_{a-1},\underbrace{t_{a+1},\ldots,t_r}_{r-a},t_a).
		\end{equation}
		For simplicity we set
		\[ \alpha_{j}=\sum_{i=1}^{j}t_{i} \ (1 \le j \le r) \ \mbox{and}\ \alpha_{j}^{\sigma}=\sum_{i=1}^{j}t_{\sigma(i)} \ (1 \le j \le r-1, \sigma \in \mathfrak{S}_{r-1}).\]
		From Proposition \ref{prop:zt-sh-formula1}, the left-hand side of \eqref{eq:gen_identity} can be computed as follows:
		\begin{align*}
		\phi \big( J(t_1,\ldots,t_r)\big)&=-\phi
		\left(\frac{x_0}{1+x_0t_r}  (y(t_1)\shuffle \cdots \shuffle y(t_{r-1}))\right)\\
		&=-\phi\left(\frac{x_0}{1+x_0t_r} \sum_{\sigma\in \mathfrak{S}_{r-1}}  y(\alpha_{r-1}^\sigma)\cdots y(\alpha_{2}^\sigma)y(\alpha_{1}^\sigma)\right)\\
		&=-\sum_{\sigma\in \mathfrak{S}_{r-1}} \phi\left(\frac{x_0}{1+x_0t_r}   y(\alpha_{r-1}) y(\alpha_{r-2}^\sigma)\cdots y(\alpha_{2}^\sigma)y(\alpha_{1}^\sigma)\right).
		\end{align*}
		Since
		\begin{align*}
		\frac{x_0}{1+x_0t_r}   y(\alpha_{r-1})  &= -\frac{1}{\alpha_r} \left( \frac{1}{1+x_0t_r}x_1 - \frac{1}{1-x_0\alpha_{r-1}}x_1\right) \\
		&= -\frac{1}{\alpha_r} (y(-t_r)-y(\alpha_{r-1})),
		\end{align*}
		again by Proposition \ref{prop:zt-sh-formula1} it can be reduced to
		\begin{align*}
		&=\frac{1}{\alpha_r} \sum_{\sigma\in \mathfrak{S}_{r-1}} \phi\left(y(-t_r)  y(\alpha_{r-2}^\sigma)\cdots y(\alpha_{2}^\sigma)y(\alpha_{1}^\sigma)\right)\\
		&-\frac{1}{\alpha_r} \sum_{\sigma\in \mathfrak{S}_{r-1}} \phi\left(y(\alpha_{r-1}^\sigma)  y(\alpha_{r-2}^\sigma)\cdots y(\alpha_{2}^\sigma)y(\alpha_{1}^\sigma)\right)\\
		&=\frac{1}{\alpha_r} \sum_{a\in[r-1]}\phi\left( y(-t_r) S_{[r-1]\setminus\{a\}}(t_1,\ldots,t_{r-1})\right) -\frac{1}{\alpha_r} \phi\left( S_{[r-1]}(t_1,\ldots,t_{r-1})\right)\\
		&=\frac{1}{\alpha_r} \sum_{a\in[r-1]}\phi\left( y(-t_r) S_{[r]\setminus\{a,r\}}(t_1,\ldots,t_{r})\right) -\frac{1}{\alpha_r} \phi\left( S_{[r]\setminus\{r\}}(t_1,\ldots,t_{r})\right).
		\end{align*}
		Proposition \ref{prop:zt-phi-formula2} for the case $u=-t_r$ and $I=[r]\setminus\{a,r\}$ gives
		\begin{align*}
		&\phi\left(  y(-t_r) S_{[r]\setminus\{a,r\}}(t_1,\ldots,t_{r})\right)\\
		&= y(-t_r)S_{[r]\setminus\{a,r\}}(t_1,\ldots,t_r)-y(t_r)\shuffle S_{[r]\setminus\{a,r\}}(t_1,\ldots,t_r)\\
		&+\sum_{b\in [r-1]\setminus\{a\}}y(-t_b)\left( y(t_r) \shuffle S_{[r]\setminus\{a,b,r\}}(t_1,\ldots,t_r)\right)\\
		&=y(-t_r)S_{[r]\setminus\{a,r\}}(t_1,\ldots,t_r)- S_{[r]\setminus\{a\}}(t_1,\ldots,t_r)\\
		&+\sum_{b\in [r-1]\setminus\{a\}}y(-t_b) S_{[r]\setminus\{a,b\}}(t_1,\ldots,t_r)\\
		&=\sum_{b\in [r]\setminus\{a\}}y(-t_b) S_{[r]\setminus\{a,b\}}(t_1,\ldots,t_r) - S_{[r]\setminus\{a\}}(t_1,\ldots,t_r).
		\end{align*}
		Likewise, Proposition \ref{prop:zt-phi-formula1} for the case $I=[r]\setminus\{r\}$ shows
		\begin{align*}
		\phi\left( S_{[r]\setminus\{r\}}(t_1,\ldots,t_{r})\right)&= S_{[r]\setminus\{r\}}(t_1,\ldots,t_{r}) -\sum_{b\in [r]\setminus\{r\}} y(-t_b)S_{[r]\setminus\{b,r\}}(t_1,\ldots,t_r).
		\end{align*}
		Therefore we have
		\begin{align*}
		\phi \big( J(t_1,\ldots,t_r)\big)&=
		\frac{1}{\alpha_r} \sum_{a\in [r-1]} \left(\sum_{b\in [r]\setminus\{a\}}y(-t_b) S_{[r]\setminus\{a,b\}}(t_1,\ldots,t_r) - S_{[r]\setminus\{a\}}(t_1,\ldots,t_r) \right)\\
		&-\frac{1}{\alpha_r} \left(S_{[r]\setminus\{r\}}(t_1,\ldots,t_{r}) -\sum_{b\in [r]\setminus\{r\}} y(-t_b)S_{[r]\setminus\{b,r\}}(t_1,\ldots,t_r)\right)\\
		&=\frac{1}{\alpha_r} \sum_{a\in [r]} \left(\sum_{b\in [r]\setminus\{a\}}y(-t_b) S_{[r]\setminus\{a,b\}}(t_1,\ldots,t_r) -S_{[r]\setminus\{a\}}(t_1,\ldots,t_r)\right)\\
		&=\sum_{a=1}^{r} J(\underbrace{t_1,\ldots,t_{a-1}}_{a-1},\underbrace{t_{a+1},\ldots,t_r}_{r-a},t_a),
		\end{align*}
		which proves \eqref{eq:gen_identity}.
		We complete the proof.
	\end{proof}

	\subsection{Remark on the value $\Omega(\kk)$}
	In the proof of \eqref{eq:omega_smzv}, we actually proved the equality
	\begin{equation*}\label{eq:Omega_formula}
	\Omega(\kk) =(-1)^{k_r} \zeta^\shuffle_{\mathcal{S}}(x_0^{k_r}(y_{k_1}\shuffle \cdots \shuffle y_{k_{r-1}}))
	\end{equation*}
	for any index $\kk=(k_1,\ldots,k_r) \in \N^r$.
	Accordingly, relations of $\Omega(\kk)$ give rise to relations of $\zeta^\shuffle_{\mathcal{S}}(\kk)$ (without taking modulo $\zeta(2)$).
	For example, from \eqref{eq:omega_formula} and \cite[(5)]{Mordell}, we have
	\begin{equation}\label{eq:omega111}
	\Omega(\{1\}^k)= -k \zeta^{MT}(\{1\}^{k-1};1)=-k! \zeta(k),
	\end{equation}
	where $\{1\}^k$ means a sequence of $1$ repeated $k$ times.
	On the other hand, it follows that
	\[ \zeta^\shuffle_{\mathcal{S}}(x_0(\underbrace{y_{1}\shuffle \cdots \shuffle y_{1}}_{k-1})) = (k-1)!\zeta^\shuffle_{\mathcal{S}}(2,\{1\}^{k-2}) .\]
	Therefore, we obtain
	\begin{equation*}\label{eq:zeta_S211}
	\zeta_{\mathcal{S}}^\shuffle (2,\{1\}^{k-2})= k\zeta(k).
	\end{equation*}

	Similarly to Conjecture \ref{conj:MT}, we observed that all multiple zeta value up to weight 12 can be written as $\Q$-linear combinations of $\Omega(\kk)$'s.

	\begin{conj}
		The space $\mathcal{Z}$ is generated by the set $\{\Omega(\kk)\mid \kk :\mbox{index}\}$.
	\end{conj}

	We further remark that Ono, Seki and Yamamoto \cite{OSY} introduced the `$t$-adic' symmetric Mordell-Tornheim multiple zeta value as a counterpart of Kamano's finite Mordell-Tornheim multiple zeta value.
	For an index ${\bf k}=(k_1,\ldots,k_r,k_{r+1})$ it is defined by
	\[\zeta^{MT}_{\widehat{\mathcal{S}}}({\bf k};t) = \lim_{M\rightarrow \infty}\sum_{i=1}^{r+1}\zeta^{MT}_{\widehat{\mathcal{S}},M}(v_i;{\bf k};t),\]
	where for $1\le i\le r$ we set
	\[ \zeta^{MT}_{\widehat{\mathcal{S}},M}(v_i;{\bf k};t)=\sum_{\substack{m_1,\ldots,m_{i-1}>0\\m_{i+1},\ldots,m_{r+1}>0\\M_{i-1}+M_{i+1,r+1}<M}} \frac{1}{(t-M_{i-1}-M_{i+1,r+1})^{k_i}(t-m_{r+1})^{k_{r+1}}}\prod_{\substack{j=1\\j\neq i}}^{r} \frac{1}{m_j^{k_j}} \]
	with $M_0=0$, $M_i=m_1+\cdots+m_i $ for $1\le i\le r+1$, $M_{i,j}=m_i+m_{i+1}+\cdots+m_j$ for $i\le j$ and $M_{i,j}=0$ if $i>j$, and also let
	\[ \zeta^{MT}_{\widehat{\mathcal{S}},M}(v_{r+1};{\bf k};t)=\sum_{\substack{m_1,\ldots,m_r>0\\M_r<M}} \frac{1}{m_1^{k_1}\cdots m_r^{k_r}M_r^{k_{r+1}}}.\]
	Letting $t\rightarrow 0$ (or taking the constant term), we see that
	\begin{align*}
	\zeta^{MT}_{\widehat{\mathcal{S}}}({\bf k};0)=(-1)^{k_{r+1}} \sum_{a=1}^{r+1}(-1)^{k_a} \zeta^{MT}(k_1,\ldots,k_{a-1},k_{a+1},\ldots,k_{r+1};k_a).
	\end{align*}
	Combining this with \eqref{eq:omega_formula}, we obtain another expression of $\Omega(\kk)$ as follows:
	\[ \Omega(k_1,\ldots,k_r)= \lim_{M\rightarrow \infty}\sum_{i=1}^{r}(-1)^{k_r}\zeta^{MT}_{\widehat{\mathcal{S}},M}(v_i;k_1,\ldots,k_r;0).\]

	\newcommand{\Cyc}{\mathcal{R}}
	\section{Further remarks}

	\subsection{Relations of $\omega_n(\kk;\zeta_n)$}
	In this subsection, we consider relations among the values $\omega_n(\kk;\zeta_n)$ which are valid for all $n$, as well as its application to relations among finite and symmetric multiple omega values.

For that purpose we consider the $\Q$-algebra
	\[\Cyc=\prod_{n \geq 1} \Q(\zeta_n).\]
	In this $\Q$-algebra, we will count the number of linearly independent relations of the form
	\begin{equation}\label{eq:rel_w}
	\sum_{\substack{\kk:{\rm index}\\ m\ge0}} a_{\kk}^{(m)} (1-\zeta_n )^m \omega_n(\kk;\zeta_n) =0 \quad (a_{\kk}^{(m)} \in \Q)
	\end{equation}
	which holds for all $n$. It is apparent that the relation \eqref{eq:rel_w} induces relations
	\[ \sum_{\kk:{\rm index}} a_{\kk}^{(0)}  \omega_{\mathcal{A}} (\kk) =0 \quad \mbox{and} \quad \sum_{\kk:{\rm index}} a_{\kk}^{(0)}  \Omega (\kk) =0,\]
	by applying Theorems \ref{thm:FMZV} and \ref{thm:SMZV}.
	Hence, finding relations of the form \eqref{eq:rel_w} will be a first step toward to capture all relations among finite and symmetric multiple omega values (note that the coefficient $a_{\kk}^{(m)}$ could be a polynomial in $n$ over $\Q$, but we get rid of such relations below).
	For an index $\kk$, we then consider a universal object
	\begin{align*}
	\omega(\kk) = \big( \omega_n(\kk ; e^{\frac{2\pi i}{n}}) \big)_{n\geq 1} \in \Cyc\,,
	\end{align*}
	which we call a cyclotomic multiple omega value.

	We give a family of relations among the cyclotomic multiple omega values.
	Set $\zeta = \big( \zeta_n \big)_{n\geq 1} \in \Cyc$.

	\begin{thm}\label{thm:cycomegarel}
		Suppose that $k_{1}, \ldots , k_{r} \ge 2$.
		Then it holds that
		\begin{align*}
		\sum_{j=1}^{r}
		\sum_{\substack{l_{1}, \ldots , l_{j-1} \ge 2 \\ l_{j}, \ldots , l_{r} \ge 1}} &
		\left\{
		\prod_{p=1}^{j-1}\binom{k_{p}-2}{l_{p}-2}
		\right\}
		\binom{k_{j}-2}{l_{j}-1}\,
		\left\{
		\prod_{p=j+1}^{r}\binom{k_{p}-1}{l_{p}-1}
		\right\} \\
		&\qquad {}\times (1-\zeta)^{\sum_{p=1}^{r}(k_{p}-l_{p})-1}
		\omega(l_{1}, \ldots , l_{r})=0.
		\nonumber
	\end{align*}
	\end{thm}

	\begin{proof} We show that  for all $n\geq 1$ and primitive roots of unity $\zeta_n$ we have
		\begin{align}
		\sum_{j=1}^{r}
		\sum_{\substack{l_{1}, \ldots , l_{j-1} \ge 2 \\ l_{j}, \ldots , l_{r} \ge 1}} &
		\left\{
		\prod_{p=1}^{j-1}\binom{k_{p}-2}{l_{p}-2}
		\right\}
		\binom{k_{j}-2}{l_{j}-1}\,
		\left\{
		\prod_{p=j+1}^{r}\binom{k_{p}-1}{l_{p}-1}
		\right\}
		\label{eq:sym-sum-rel} \\
		&\qquad {}\times (1-\zeta_{n})^{\sum_{p=1}^{r}(k_{p}-l_{p})-1}
		\omega_{n}(l_{1}, \ldots , l_{r}; \zeta_{n})=0.
		\nonumber
	\end{align}
		For $m_{1}, \ldots , m_{r} \ge 1$ it holds that
		\begin{align*}
		\sum_{j=1}^{r}q^{m_{1}+\cdots +m_{j-1}}[m_{j}]=[m_{1}+\cdots +m_{r}].
		\end{align*}
		Suppose that $\sum_{j=1}^{r}m_{j}=n$.
		Divide the both sides with $\prod_{p=1}^{r}[m_{p}]^{k_{p}}$ and
		set $q=\zeta_{n}$.
		Then we have
		\begin{align*}
		\sum_{j=1}^{r}\prod_{p=1}^{j-1}\frac{q^{m_{p}}}{[m_{p}]^{k_{p}}}
		\frac{1}{[m_{j}]^{k_{j}-1}}
		\prod_{p=j+1}^{r}\frac{1}{[m_{p}]^{k_{p}}}\biggl|_{q=\zeta_{n}}=0.
		\end{align*}
		Now the desired equality follows from
		\begin{align*}
		&
		\frac{q^{m}}{[m]^{k}}=\sum_{l \ge 2}\binom{k-2}{l-2}(1-q)^{k-l}F_{l}(m) \qquad (k \ge 2),  \\
		&
		\frac{1}{[m]^{k}}=\sum_{l \ge 1}\binom{k-1}{l-1}(1-q)^{k-l}F_{l}(m) \qquad (k \ge 1).
		\end{align*}
		We complete the proof.
	\end{proof}
	For example, taking $k_1=k_2=\dots=k_r=2$ we obtain for all $r\geq 2$
	\begin{align}\label{eq:222case}
	\sum_{j=1}^r \binom{r}{j} (1-\zeta)^{j-1} \omega(\{2\}^{r-j},\{1\}^j) = 0\,.
	\end{align}

	Using Theorems \ref{thm:FMZV} and \ref{thm:SMZV}, we obtain the following.

	\begin{cor}\label{cor:valance_formula}
		Suppose that $k_{1}, \ldots , k_{r} \ge 2$.
		Then we have
		\begin{align}
		\notag \sum_{j=1}^{r}\omega_{\mathcal{A}}(k_{1}, \ldots , k_{j}-1, \ldots , k_{r})=0  ,\\
		\label{eq:rel_Omega}\sum_{j=1}^{r}\Omega(k_{1}, \ldots , k_{j}-1, \ldots , k_{r})=0.
		\end{align}
	\end{cor}

	\begin{proof}
		The first equation is an immediate from  \eqref{eq:sym-sum-rel} and Theorem \ref{thm:FMZV}.
		For the second equation, set $\zeta_{n}=e^{2\pi i/n}$ in \eqref{eq:sym-sum-rel},
		and take the limit $n \to 0$.
		Then the terms which satisfy $l_{p}=k_{p} \, (p\not=j)$ and $l_{j}=k_{j}-1$ only remain,
		and by Theorem \ref{thm:SMZV} we get the desired formula.
	\end{proof}

	Remark that the relation \eqref{eq:rel_Omega} can be also deduced from the following relations of the Mordell-Tornheim multiple zeta values.
	For $k_1,\ldots,k_r,l\ge2$ we have
	\begin{align*}
	\zeta^{MT}(k_1,\ldots,k_r;l-1)&=\sum_{m_1,\ldots,m_r\ge1} \frac{m_1+\cdots+m_r}{m_1^{k_1}\cdots m_r^{k_r} (m_1+\cdots+m_r)^l}\\
	&=\sum_{a=1}^r \zeta^{MT}(k_1,\ldots,k_a-1,\ldots,k_r;l),
	\end{align*}
	which is found in the literature (see e.g. the proof of \cite[Proposition 4.16]{Onodera}).

	\subsection{Observation on cyclotomic multiple omega values}

	In this subsection, we work on the $\Q$-vector subspace $\MTW_{k}$ of $\Cyc$ defined for $k\ge1$ by
	\begin{align*}
	\MTW_{k} =\big\langle (1-\zeta)^m \omega(k_1,\dots,k_r) \mid 0 \leq m < k,\, 2\leq r \leq k,\, k_1+\dots + k_r + m = k \big\rangle_{\Q} .
	\end{align*}
	In the following we discuss the relations among the generators of $\MTW_k$ and its dimension. From Theorems \ref{thm:FMZV} and \ref{thm:SMZV}, one obtains two well-defined $\Q$-linear maps $\varphi^{\bullet}_k : \MTW_{k} \rightarrow \MTWB_k $ given for $\bullet \in \{\mathcal{A}, \mathcal{S} \}$ and each generator by
	\begin{align*}
	\varphi_k^\bullet\big((1-\zeta)^m \omega(\kk) \big)=\begin{cases} \omega_{\bullet}(\kk) & \mbox{if}\ m=0 \\
	0 & \mbox{if}\ m>0
	\end{cases}\,.
	\end{align*}
	By definition these maps are surjective, so we get
	\begin{center}
		\begin{tikzcd}
		& \MTW_{k} \arrow[twoheadrightarrow]{dr}{\varphi^{\mathcal{S}}_k} \arrow[twoheadrightarrow]{dl}[swap]{\varphi^{\mathcal{A}}_k}  \\
		\MTWA_{k} \arrow[dotted]{rr}{\varphi_k} && \MTWS_{k}
		\end{tikzcd},
	\end{center}
	where the map $\varphi_k: \MTWA_k\rightarrow \MTWS_k$ denotes the conjectured isomorphism of Conjecture \ref{conj:main}.
	This picture suggests that $\ker \varphi^{\mathcal{A}}_k  \overset{?}{=} \ker \varphi^{\mathcal{S}}_k$. Clearly, we have $(1-\zeta) \MTW_{k-1} \subset \ker \varphi^{\bullet}_k$ for $\bullet \in \{\mathcal{A}, \mathcal{S} \}$, but there are also non-trivial elements in $\ker \varphi^{\bullet}_k$, since Theorem \ref{thm:cycomegarel} gives relations among elements in $\MTW_{k}$. Numerically get the following relations.\\
	weight $3$:
	\begin{align*}
	\omega(2,1) &=-\frac{1}{2}(1-\zeta)\omega(1,1) \,.
	\end{align*}
	weight $4$:
	\begin{align*}
	\omega(3,1) &\stackrel{?}{=} \omega(2,1,1) + \frac{1}{2} (1-\zeta)\omega(1,1,1)+\frac{1}{4} (1-\zeta)^2\omega(1,1) \,,\\
	\omega(2,2) &\stackrel{?}{=} -\omega(2,1,1) - \frac{1}{2} (1-\zeta)\omega(1,1,1) +\frac{1}{4} (1-\zeta)^2\omega(1,1)\,.
	\end{align*}
	weight $5$:
	\begin{align*}
	\omega(4,1) &\stackrel{?}{=} -\frac{3}{2}(1-\zeta)\omega(2,1,1) - \frac{3}{4} (1-\zeta)^2\omega(1,1,1)-\frac{1}{8} (1-\zeta)^3\omega(1,1) \,,\\
	\omega(3,2) &\stackrel{?}{=} \frac{1}{2}(1-\zeta)\omega(2,1,1) + \frac{1}{4} (1-\zeta)^2\omega(1,1,1)-\frac{1}{8}(1-\zeta)^3 \omega(1,1)\,,\\
	\omega(2,2,1) &= -(1-\zeta)\omega(2,1,1) - \frac{1}{3} (1-\zeta)^2\omega(1,1,1) \,.
	\end{align*}
Notice that the first and the last relation are consequences of \eqref{eq:222case}.
	Remember that $\omega(k_{\sigma(1)},\ldots,k_{\sigma(r)})=\omega(k_1,\ldots,k_r)$ holds for any permutation $\sigma\in \mathfrak{S}_r$.
	The above relations conjecturally give all relations in each weight. Doing numerical calculation until weight $11$, we get the following table:
	\begin{figure}[h!]
		\begin{tabular}{|c|ccccccccccc|}
			\hline
			$k$ & 1 & 2 & 3 &  4&  5& 6 &7  &8  &9  & 10  & 11 \\ [2pt] \hline
			$\dim \MTW_{k}  \overset{?}{=}$ &0  & 1 & 2 & 4 & 7 & 12 & 19 & 30 & 45 & 68 & 99    \\[2pt] \hline
			$\dim \MTW_{k} \slash (1-\zeta)\MTW_{k-1}  \overset{?}{=}$       & 0 &1  & 1 & 2 &3  &  5&  7& 11 & 15 &23 & 31   \\[2pt] \hline
		\end{tabular}
	\end{figure}

	Notice that the numbers in the second row are given by the number of partitions of $k-2$ until weight $k=9$. Until this weight it seems that the elements of the form $	(1-\zeta)^m \omega(k_1,\dots,k_r,1,1) $ with $m\geq 0,\,r\geq 0,\, k_1\geq \ldots \geq k_r\geq 1$ and $k_1+\dots+k_r+m= k-2$ form a basis. But starting in weight $10$ there seem to be elements, which are not linear combinations of these.

%

	In contrast to the $z_n(\kk ;q)$ introduced in \cite{BTT18} (see \eqref{eq:def_znq}) the algebraic structure of the cyclotomic multiple omega values is not known yet. We are  expecting that $\MTW_{k_1} \MTW_{k_2} \subset \MTW_{k_1 + k_2}$ holds for $k_1,k_2\geq 2$. Namely, the space $\MTW=\sum_{k\ge0} \MTW_k$ forms a $\Q$-subalgebra of $\Cyc$, where $\MTW_0=\Q$. For example, one can check numerically that
	\begin{align}\begin{split}\label{eq:cycomegaproducts}
		\omega(1,1)^2 \stackrel{?}{=} &-5 \omega(2,1,1) - \frac{5}{2} (1-\zeta)\omega(1,1,1) + \frac{1}{4} (1-\zeta)^2 	\omega(1,1)\,,\\
		\omega(1,1) \omega(1,1,1) \stackrel{?}{=} &-2 \omega(2,1,1,1)-3 \omega(3,1,1) - (1-\zeta) \omega(1,1,1,1) \\
		&- 3 (1-\zeta)\omega(2,1,1) - \frac{1}{3} (1-\zeta)^2 \omega(1,1,1) \,.
\end{split}
	\end{align}

 Since $\varphi^\bullet_{\wt(\kk)+\wt({\bf l})}(\omega(\kk) \omega({\bf l}))= \omega_\bullet(\kk)\omega_\bullet({\bf l}) $ and $\omega_{\bullet}(1,1)=0$  for $\bullet\in\{\mathcal{A},\mathcal{S}\}$ (see Section \ref{subsec:specialval} below), the equations \eqref{eq:cycomegaproducts} would imply the relations
\begin{align*}
\omega_{\bullet}(2,1,1)\stackrel{?}{=}0,\qquad   2\omega_{\bullet}(2,1,1,1) + 3 \omega_{\bullet}(3,1,1)\stackrel{?}{=}0\,.
\end{align*}
	\subsection{Special values}\label{subsec:specialval}
	In this subsection, we make a list of known values of finite and symmetric multiple omega values.
	\begin{enumerate}
	\item
	From Theorem \ref{thm:omega_fsmzv}, for $\bullet\in\{\mathcal{A},\mathcal{S}\}$ and $k_1,k_2\ge1$ we obtain $ \omega_{\bullet}(k_1,k_2) =(-1)^{k_2} \zeta_{\bullet}(k_1+k_2)$. Since $\zeta_\bullet(k)=0$, we see that
	\[\omega_{\bullet}(k_1,k_2)=0.\]
	\item
	Letting $k_1=\cdots=k_r=2$ in Corollary \ref{cor:valance_formula}, we also have
	\[ \omega_{\bullet}(\{2\}^{r-1},1)=0.\]
	\item
	By \eqref{eq:omega111}, for $k\ge2$ we see that
	\[ \omega_{\mathcal{S}} (\{1\}^k)\equiv  -k! \zeta(k) \mod \zeta(2)\mathcal{Z}.\]
	This also follows from \eqref{eq:omega_smzv} and a special case of the sum formula for symmetric multiple zeta values by Murahara \cite[Theorem 1.2]{Murahara}.
	Likewise, from \eqref{eq:omega_fmzv} and a special case of the sum formula for finite multiple zeta values by Saito and Wakabayashi \cite{SaitoWakabayashi}, for $k\ge2$ we have
	\[ \omega_{\mathcal{A}} (\{1\}^k)=  -k! Z(k) ,\]
	where $Z(k)$ is the `true' analogue of $\zeta(k)$ in $\mathcal{A}$ defined for $k\ge1$ by
	\[ Z(k)=\left( \frac{B_{p-k}}{k} \mod p \right)_p \in \mathcal{A}\qquad (B_{p-k}:\mbox{Bernoulli number}).\]
	So far, we were not able to find explicit formulas for $\omega_{\bullet}(\{a\}^k)$ (note that $\zeta_{\bullet}(\{a\}^k)=0$ for any $a,k\in \N$ and $\bullet\in\{\mathcal{A},\mathcal{S}\}$).
	\end{enumerate}



\end{document}